\documentclass[11pt]{amsart} \textwidth=14.5cm \oddsidemargin=1cm
\usepackage{amsmath,wasysym}
\usepackage{upgreek}
\usepackage{amsxtra}
\usepackage{amscd}
\usepackage{amsthm}
\usepackage{amsfonts}
\usepackage{amssymb}
\usepackage{eucal}
\usepackage{graphics, color}
\usepackage{hyperref,mathrsfs}
\usepackage[usenames,dvipsnames]{xcolor}
\usepackage{tikz}
\usepackage{stmaryrd}
\usepackage[all]{xy}
\usepackage{hyperref}
\usepackage{color}
\usepackage{bbm} 
\allowdisplaybreaks

\hypersetup{colorlinks,linkcolor=blue, urlcolor=blue,
citecolor=blue}

\newtheorem{thm}{Theorem} [section]
\newtheorem{lem}[thm]{Lemma}
\newtheorem{cor}[thm]{Corollary}
\newtheorem{prop}[thm]{Proposition}

\theoremstyle{definition}

\newtheorem{rem}[thm]{Remark}

\numberwithin{equation}{section}

\newcommand{\V}{\mathbb{V}}

\newcommand{\inner}[1]{\langle #1 \rangle}

\title[Super-extensions of tensor algebras and their applications]
{Super-extensions of tensor algebras and their applications}

\author[Run-Qiang Jian]{Run-Qiang Jian}
\address{School of computer science and technology, Dongguan University of Technology, Dongguan 523808, China}
\email{jianrq@dgut.edu.cn (Jian)}

\author[Xianfa Wu]{Xianfa Wu}

\address{School of Mathematical Sciences, East China Normal University, Shanghai 200241, China}
\email{52275500051@stu.ecnu.edu.cn (Wu)}

\begin{document}

\begin{abstract}
  We construct a super-extension of the usual $q$-tensor algebra through super-actions of Hecke algebras.
  A double centralizer property is established on this extension space, which provides a modest generalization of the quantum Schur-Sergeev duality.
  As an application, we present a new proof of the quantum Schur-Sergeev duality.
\end{abstract}

\maketitle
\setcounter{tocdepth}{1}

\section{Introduction}
Denote by $V^{\otimes d}$ the $d$-fold of a vector space $V$.
In \cite{Sc27}, Schur studied the relation between the permutation action of the symmetric group $\mathfrak{S}_d$ and the diagonal action of the general linear group $\mathrm{GL}(V)$ on $V^{\otimes d}$. He pointed out that the actions of these two groups on $V^{\otimes d}$ admit a double centralizer property, which is just the celebrated Schur-Weyl duality.
Since its importance in representation theory, this duality has been generalized in several directions during the past decades.
In the late 1980s, the Schur-Sergeev duality, which is a super version of the Schur-Weyl duality, was established by Sergeev \cite{Ser85} (also by Berele and Regev \cite{BR87} independently). Meanwhile, Jimbo \cite{Jim86} introduced a duality between Hecke algebras and general linear quantum groups, which is called the Schur-Jimbo duality nowadays.
In \cite{Mi06}, Mitsuhashi provided a super quantized version of the Schur-Weyl duality, which can be regarded as a common generalization of the dualities due to Sergeev and Jimbo.

Usually, in order to prove these dualities, people use the double commutant theorem. For example, for the classical Schur-Weyl duality, one needs to show that $\mathbb{C}\mathrm{GL}(V)=\mathrm{End}_{\mathbb{C}\mathfrak{S}_d}(V^{\otimes d})$. Then by the semisimplicity of the group algebra $\mathbb{C}\mathfrak{S}_d$ one has $\mathbb{C}\mathfrak{S}_d=\mathrm{End}_{\mathbb{C}\mathrm{GL}(V)}(V^{\otimes d})$. The verification of the inclusion $\mathbb{C}\mathrm{GL}(V)\subset\mathrm{End}_{\mathbb{C}\mathfrak{S}_d}(V^{\otimes d})$ is trivial. But the reverse direction requires further non-trivial techniques. We refer the reader to \cite{GW03} and \cite{CW12} for detailed proof.

In \cite{It12}, Itoh generalized the usual tensor algebra $T(V)$ by considering actions of symmetric groups.
He studied operators of left multiplication by vectors and linear functionals, and established a duality theorem between the algebra generated by these operators and the infinity symmetric group.
As a consequence, he provided a new proof of the Schur-Weyl duality theorem.
In a subsequent paper \cite{It15}, he generalized his construction to quantum case in which symmetric group actions are replaced by Hecke algebra actions.
Using a similar argument, he gave a new proof of Schur-Jimbo duality.

The main purpose of this paper is to generalize Itoh's results to the quantum super case.
First, we establish the relevant construction for the super case and obtain corresponding results.
Then, based on the special challenges of the quantum situation,
we achieve the quantum super case through a generally similar but slightly different approach.
For the $q$-tensor superspace $\V^{\otimes d}$, we introduce a map $\gamma$ to simplify
and unify the parity calculation problems that arise in the action of the Hecke algebra $\mathscr{H}_d$.
Next, we construct the quantum super-extension algebra
\[
  \widetilde{T}(\V)=\bigoplus_{d\geq 0}\V^{\otimes d}\otimes_{\mathscr{H}_d}\mathscr{H}_{\infty},
\]
and define operators $L(\phi)$, $L(v^\ast)$ over it.
Following this, we carefully verify the well-definedness of the relevant operators and compute their commutation relations.
Finally, in Theorem ~\ref{g-duality} we demonstrate a Schur-Weyl type duality
\[
  \mathcal{A}_d^k \curvearrowright \V^{\otimes d} \otimes_{\mathscr{H}_d} \mathscr{H}_{d+k} \curvearrowleft \mathscr{H}_{d+k},
\]
where $\mathcal{A}_d^k$ is the subalgebra of $\mathrm{End}_{\mathbb{C}(q)}(\V^{\otimes d} \otimes_{\mathscr{H}_d} \mathscr{H}_{d+k})$
spanned by operators of the form
\[
  L(w_d)\cdots L(w_1)L(\sigma)L(v_1^\ast)\cdots L(v_d^\ast)
\]
with $w_1,\ldots,w_d\in \V$, $\sigma\in \mathscr{H}_{d+k}$,  and $v_1^*,\ldots,v_k^* \in \V^*$.
In the special case of $k=0$, we recover Mitsuhashi's duality, and Schur-Sergeev duality is a corollary for the non-quantum case.

This paper is organized as follows.
In Sections 2 - 4, we construct the super-extension algebra, study the operators defined over it and prove the duality theorems for super case finally.
In Sections 5 - 7, we deal with the quantum super version.

\subsection*{Acknowledgement}
We would like to thank Li Luo for many helpful discussions.
Especially, he introduced the Schur-Sergeev duality and its quantum version to us which initiates this work.
We would also like to thank the referee for thorough review and helpful suggestions.

\section{Super-extensions through symmetric groups}

\subsection{Symmetric groups}
For any positive integer $d$, we denote by $\mathfrak{S}_d$ the symmetric group of the set $[d]=\{1,2,\ldots,d\}$, and by $s_i$, $1\leq i\leq d-1$, the transposition interchanging $i$ and $i+1$. We often use cycle notation to denote a permutation. So $\tau=(i_1\ i_2\ \cdots\ i_k)\in \mathfrak{S}_d$ maps $i_r$ to $i_{r+1}$ for $1\leq r\leq k-1$ and $i_k$ to $i_1$,  and fixes other numbers. There is a natural sequence of inclusions: $\mathfrak{S}_1\subset\cdots\subset \mathfrak{S}_d\subset \cdots$. The direct limit is denoted by $\mathfrak{S}_\infty$. One can view $\mathfrak{S}_\infty$ as the group generated by $\{s_i\}_{i\geq 1}$ subject to the relations:
$$s_i^2=1,$$  $$s_is_{i+1}s_i=s_{i+1}s_is_{i+1},$$ and $$s_is_j=s_js_i, \text{ whenever }|i-j|>1.$$ Given $k\geq 1$, we define a group homomorphism $ ^{\uparrow k}:\mathfrak{S}_\infty\rightarrow \mathfrak{S}_\infty$ by $s_i^{\uparrow k}=s_{i+k}$.

\subsection{Superspaces and general linear Lie superalgebras}
Let $m,n$ be two fixed positive integers, and $V=\mathbb{C}^{m|n}=\mathbb{C}^m\oplus \mathbb{C}^n$.
As usual, we set $V_{\overline{0}}=\mathbb{C}^m$ and $V_{\overline{1}}=\mathbb{C}^n$.
Vectors in $V_{\overline{0}}$ (resp. $V_{\overline{1}}$) are of parity $0$ (resp. parity $1$).
The parity of a homogeneous element $v\in V$ is denoted by $|v|$.
In this section, we will denote by $e_i$, $1\leq i\leq m+n$, the vector in $V$ that is 1 in the $i$-th position and 0 elsewhere.
Then $\{e_1,\ldots, e_m\}$ and $\{e_{m+1},\ldots, e_{m+n}\}$ are bases of $V_{\overline{0}}$ and $V_{\overline{1}}$, respectively.
We denote $\{e_1^\ast,\ldots, e_{m+n}^\ast\}$ as the dual basis of $\{e_1,\ldots, e_{m+n}\}$ in the dual space $V^\ast$, i.e., $e_i^\ast(e_j)=\delta_{ij}$. Here $\delta_{ij}$ is the Kronecker symbol. Define a parity function on this dual basis by requiring that $|e_i^\ast|=|e_i|$ for each $1\leq i\leq m+n$. Thus the dual space $V^\ast=V_{\overline{0}}^\ast\oplus V_{\overline{1}}^\ast$ of $V$ is a $\mathbb{Z}_2$-graded vector space.

The general linear Lie superalgebra $\mathfrak{gl}(m|n)$ is
 $$\bigg\{\left(\begin{array}{cc}
A&B\\
C&D
\end{array}\right)\bigg|~A\in M_{m\times m}(\mathbb{C}),B\in M_{m\times n}(\mathbb{C}),C\in M_{n\times m}(\mathbb{C}),D\in M_{n\times n}(\mathbb{C})\bigg\}.$$ We can see that $$\mathfrak{gl}(m|n)=\mathfrak{gl}(m|n)_{\overline{0}}\oplus \mathfrak{gl}(m|n)_{\overline{1}},$$where $$\mathfrak{gl}(m|n)_{\overline{0}}=\bigg\{\left(\begin{array}{cc}
A&0\\
0&D
\end{array}\right)\bigg|~A\in M_{m\times m}(\mathbb{C}),D\in M_{n\times n}(\mathbb{C})\bigg\},$$ and$$\mathfrak{gl}(m|n)_{\overline{1}}=\bigg\{\left(\begin{array}{cc}
0&B\\
C&0
\end{array}\right)\bigg|~B\in M_{m\times n}(\mathbb{C}),C\in M_{n\times m}(\mathbb{C})\bigg\}.$$The parity of a homogeneous element $g\in \mathfrak{gl}(m|n)$ is denoted by $|g|$. As usual, for any $1\leq i,j \leq m+n $, let $E_{ij}$ be the $(m+n)\times (m+n)$ matrix with $1$ in the $(i,j)$-position and $0$ elsewhere. Then $|E_{ij}|=|e_i|+|e_j^\ast|$.

\subsection{Representations on $V^{\otimes d}$}
For further discussion, we make some conventions. From now on, elements in $V$ or $\mathfrak{gl}(m|n)$ appeared in our paper are always assumed to be homogeneous unless otherwise specified. Subscripts are often arranged in a decreasing order. For example, a pure $d$-tensor of $V$ is written as $v_d\otimes\cdots\otimes v_1$.

The general linear Lie superalgebra $\mathfrak{gl}(m|n)$ acts from left on $V$ in a natural way. This action induces a representation $\rho_d$ of the enveloping algebra $U(\mathfrak{gl}(m|n))$ on $V^{\otimes d}$ for each $d\geq 1$ as follows: $$\rho_d(g)(v_d\otimes \cdots\otimes v_1)=\sum_{k=1}^d(-1)^{|g|(|v_d|+\cdots+|v_{k+1}|)}v_d\otimes\cdots\otimes gv_k\otimes \cdots\otimes v_1,$$ where $g\in \mathfrak{gl}(m|n)$ and $v_1,\ldots,v_d\in V$.

On the other hand, the symmetric group $\mathfrak{S}_d$ acts from right on $[m+n]^d = \{(i_d,\ldots,i_1)|1\leq i_1,\ldots,i_d\leq m+n\}$
as follows: for each $1\leq k<d$ and $I=(i_d,\ldots,i_1) \in [m+n]^d$,
\[I \cdot s_k=(i_d,\ldots,i_{k},i_{k+1},\ldots,i_1).\]
Thus $I \cdot\sigma=(i_{\sigma(d)},\ldots,i_{\sigma(1)})$ for $\sigma\in \mathfrak{S}_d$. We then have a right $\mathfrak{S}_d$-action on $V^{\otimes d}$ defined by
\[
  (v_d\otimes \cdots\otimes v_1) \cdot s_i=(-1)^{|v_i|\cdot|v_{i+1}|}v_d\otimes\cdots\otimes v_i\otimes v_{i+1}\otimes \cdots\otimes v_1.
\]
For example, if $\sigma=(123)=s_1s_2$, then
\[
  (v_3\otimes v_2\otimes v_1) \cdot \sigma = (-1)^{|v_2|\cdot |v_1|+|v_3|\cdot |v_1|}v_1\otimes v_3\otimes v_2.
\]
In general, for any $\sigma\in \mathfrak{S}_d$, \begin{equation}\label{Action of permutation}(v_d\otimes v_{d-1}\otimes \cdots\otimes v_1) \cdot \sigma=\varepsilon_{v_d,\ldots,v_1;\sigma}v_{\sigma(d)}\otimes v_{\sigma(d-1)}\otimes \cdots\otimes v_{\sigma(1)},\end{equation}where $\varepsilon_{v_d,\ldots,v_1;\sigma}$ is either $1$ or $-1$, depending on $v_d,\ldots,v_1$ and $\sigma$.

\subsection{Super-extensions}
From now on, we will write a pure tensor $v_d\otimes \cdots\otimes v_1\in V^{\otimes d}$ shortly as $v_d\cdots v_1$. We denote $$\overline{T}_d(V)=V^{\otimes d}\otimes_{\mathbb{C}\mathfrak{S}_d} \mathbb{C}\mathfrak{S}_\infty,$$ and $$\overline{T}(V)=\bigoplus_{d\geq 0}\overline{T}_d(V).$$ We define an associative product $\cdot$ on $\overline{T}(V)$ as follows: for any $u_k \cdots   u_1\otimes \tau\in \overline{T}_k(V)$ and $v_d \cdots   v_1\otimes \sigma\in \overline{T}_d(V)$, $$(u_k \cdots  u_1\otimes \tau)\cdot(v_d \cdots  v_1\otimes \sigma)=u_k \cdots   u_1  v_d \cdots v_1\otimes \tau^{\uparrow d}\sigma.$$ For any $v,w\in V$, we have immediately that \begin{equation*}\label{Relation 1}(v w)\cdot s_1=(-1)^{|v|\cdot|w|}w v,\end{equation*} and \begin{equation*}s_i\cdot v=v\otimes s_{i+1}.\end{equation*} For any $\varphi\in \overline{T}(V)$, we define \[\begin{array}{cccc}
L(\varphi): &\overline{T}(V) &\rightarrow& \overline{T}(V),\\[3pt]
& \psi&\mapsto&\varphi\cdot \psi.
\end{array}\]
For a homogeneous element $v^\ast \in V^\ast$, we define $L(v^\ast):\overline{T}(V)\rightarrow \overline{T}(V)$ by \begin{align*}
\lefteqn{L(v^\ast)(v_d  \cdots  v_1\otimes \sigma)}\\[3pt]
&=\sum_{k=1}^d(-1)^{|v^\ast|(|v_d|+\cdots+|v_{k+1}|)}\inner{v^\ast, v_k}v_d \cdots  \hat{v}_k \cdots  v_1\otimes (d\ d-1\ \cdots\ k+1\ k)\sigma\\[3pt]
&=\sum_{k=1}^d\Big((v^\ast\otimes \mathrm{id}_{V^{\otimes d-1}})\big((v_d \cdots  v_1) \cdot (s_k\cdots s_{d-2}s_{d-1})\big)\Big)\otimes s_{d-1}s_{d-2}\cdots s_k\sigma,
\end{align*}where the symbol $\hat{v}_k$ means the omission of $v_k$.
The well-definedness of these two maps can be checked directly and stated formally as:

\begin{prop}\label{Commutant relation}The operators $L(\sigma)$, $L(e_i)$, and $L(e_i^\ast)$ for $\sigma\in \mathfrak{S}_\infty, 1\leq i\leq m+n$ are elements in $\mathrm{End}_{\mathbb{C}\mathfrak{S}_\infty}(\overline{T}(V))$, where $\mathfrak{S}_\infty$ acts on $\overline{T}(V)$ by right multiplication.
\end{prop}

\section{Properties of operators}
It is easy to verify the following commutation relations.

\begin{prop}\label{Relation 2}For any $1\leq i,j\leq m+n$ and $\sigma\in \mathfrak{S}_\infty$, we have\begin{equation*}L(e_i)L(e_j)=(-1)^{|e_i|\cdot |e_j|}L(e_j)L(e_i)L(s_1),\end{equation*} \begin{equation*}L(e_i^\ast)L(e_j^\ast)=(-1)^{|e_i^
\ast|\cdot |e_j^\ast|}L(s_1)L(e_j^\ast)L(e_i^\ast),\end{equation*}\begin{equation*}L(e_j^\ast)L(e_i)=(-1)^{|e_j|\cdot |e_i|}L(e_i)L(s_1)L(e_j^\ast)+\langle e_j^\ast,e_i\rangle,\end{equation*}
\begin{equation*}L(\sigma)L(e_i)=L(e_i)L(\sigma^{\uparrow 1}),\end{equation*} and \begin{equation*}L(e_i^\ast)L(\sigma)=L(\sigma^{\uparrow 1})L(e_i^\ast).\end{equation*}
\end{prop}

The action of $U(\mathfrak{gl}(m|n))$ on $V^{\otimes d}$ can be naturally extended to the action on $\overline{T}_d(V)$,
and it can be expressed in terms of $L(e_i)$ and $L(e_i^\ast)$. By direct calculation, we have
\begin{align*}
  \lefteqn{L(e_i)L(e_j^\ast)(v_d \cdots v_1\otimes \sigma)}\\[3pt]
  &=\sum_{k=1}^d(-1)^{|e_j^\ast|(|v_d|+\cdots+|v_{k+1}|)}\inner{e_j^\ast, v_k}e_i  v_d \cdots  \hat{v}_k \cdots  v_1\otimes s_{d-1}s_{d-2}\cdots s_k\sigma\\[3pt]
  &=\sum_{k=1}^d(-1)^{(|e_i|+|e_j^\ast|)(|v_d|+\cdots+|v_{k+1}|)}\inner{e_j^\ast, v_k} v_d \cdots  v_{k+1}  e_i  v_{k-1} \cdots  v_1\otimes \sigma\\[3pt]
  &=\sum_{k=1}^d(-1)^{(|e_i|+|e_j^\ast|)(|v_d|+\cdots+|v_{k+1}|)} v_d \cdots  v_{k+1}  \inner{e_j^\ast, v_k}e_i  v_{k-1} \cdots v_1\otimes \sigma\\[3pt]
  &=\sum_{k=1}^d(-1)^{|E_{ij}|(|v_d|+\cdots+|v_{k+1}|)} v_d \cdots  v_{k+1}(E_{ij} v_k ) v_{k-1} \cdots  v_1\otimes \sigma\\[3pt]
  &=\big(\rho_d(E_{ij})(v_d  \cdots  v_1)\big)\otimes \sigma.
\end{align*}
In particular, on $V^{\otimes d}$, we have \begin{equation}\label{Eij}\rho_d(E_{ij})=L(e_i)L(e_j^\ast).\end{equation}

The stabilizer of $I \in [m+n]^d$ is denoted by $(\mathfrak{S}_d)_I$, i.e., $(\mathfrak{S}_d)_I=\{\sigma\in \mathfrak{S}_d|I.\sigma=I\}$. For $1\leq k\leq m+n$, denote $$\mathbf{m}_k(I)=\sharp\{j\in[d]~|~i_j=k\},$$ $$\mathbf{m}(I)=(\mathbf{m}_1(I),\ldots,\mathbf{m}_{m+n}(I)),$$ and $$\mathbf{m}(I)!=\mathbf{m}_1(I)!\cdots \mathbf{m}_{m+n}(I)!.$$For example, $\mathbf{m}_1(3,2,3)=0$, $\mathbf{m}_3(3,2,3)=2$, and $\mathbf{m}(3,2,3)=(0,1,2,0,\ldots,0)$. Obviously, one has $|(\mathfrak{S}_d)_I|=\mathbf{m}(I)!$.

\begin{lem}\label{Lemma 1}Set \begin{align*}
\mathcal{S}_1&=\{1,s_1\},\\[3pt]
\mathcal{S}_2&=\{1,s_2,s_2s_1\},\\[3pt]
\mathcal{S}_3&=\{1,s_3,s_3s_2,s_3s_2s_1\},\\[3pt]
&\vdots\\[3pt]
\mathcal{S}_{d-1}&=\{1,s_{d-1},s_{d-1}s_{d-2},\ldots,s_{d-1}s_{d-2}\cdots s_2s_1\}.
\end{align*}Then every $\sigma\in \mathfrak{S}_d$ can be written uniquely as $$\sigma=\sigma_1\sigma_2\cdots\sigma_{d-1},$$ with $\sigma_i\in \mathcal{S}_i$. Moreover, the above expression is reduced.\end{lem}
\begin{proof}For a proof, one can consult \cite[Corollary~4.4]{KT08}.\end{proof}

\begin{prop}\label{New expression}For $I=(i_d,\ldots,i_1)\in [m+n]^d$ and $\tau\in \mathfrak{S}_d$, we have \begin{equation*}L(e_{i_{\tau(d)}})\cdots L(e_{i_{\tau(1)}})L(e_{i_{\tau(1)}}^\ast)\cdots L(e_{i_{\tau(d)}}^\ast)(e_{i_d}  \cdots  e_{i_1})=\mathbf{m}(I)!e_{i_d}  \cdots  e_{i_1}.\end{equation*}\end{prop}
\begin{proof}We have\begin{align*}
\lefteqn{L(e_{i_{\tau(1)}}^\ast)\cdots L(e_{i_{\tau(d)}}^\ast)(e_{i_d}  \cdots  e_{i_1})}\\
&=L(e_{i_{\tau(1)}}^\ast)\cdots L(e_{i_{\tau(d-1)}}^\ast)(e_{i_{\tau(d)}}^\ast\otimes \mathrm{id}_{V^{\otimes d-1}})\Bigg(\sum_{\sigma_{d-1}\in \mathcal{S}_{d-1}}(e_{i_d} \cdots  e_{i_1}).\sigma_{d-1}^{-1}\otimes \sigma_{d-1}\Bigg)\\[3pt]
&=\sum_{\sigma_{d-1}\in \mathcal{S}_{d-1}}L(e_{i_{\tau(1)}}^\ast)\cdots L(e_{i_{\tau(d-2)}}^\ast)\\
&\ \ \ \ \ \ \ \ \ \ \ \  (e_{i_{\tau(d)}}^\ast\otimes e_{i_{\tau(d-1)}}^\ast\otimes \mathrm{id}_{V^{\otimes d-2}})\Bigg(\sum_{\sigma_2\in \mathcal{S}_{d-2}}\big((e_{i_d} \cdots  e_{i_1}).\sigma_{d-1}^{-1}\big).\sigma_{d-2}^{-1}\otimes \sigma_{d-2}\sigma_{d-1}\Bigg)\\[3pt]
&=\cdots\\[3pt]
&=\sum_{\sigma_1\in \mathcal{S}_{1}}\cdots\sum_{\sigma_{d-1}\in \mathcal{S}_{d-1}}(e_{i_{\tau(d)}}^\ast  \cdots   e_{i_{\tau(1)}}^\ast)\big((e_{i_d}  \cdots  e_{i_1}).(\sigma_{d-1}^{-1}\cdots\sigma_{1}^{-1})\big)\otimes \sigma_{1}\cdots\sigma_{d-1}\\[3pt]
&=\sum_{\sigma\in \mathfrak{S}_d}(e_{i_{\tau(d)}}^\ast  \cdots  e_{i_{\tau(1)}}^\ast)\big((e_{i_d}  \cdots   e_{i_1}).\sigma^{-1}\big)\otimes \sigma\\[3pt]
&=\sum_{\sigma\in \mathfrak{S}_d}\varepsilon_{e_{i_d},\ldots,e_{i_1};\sigma^{-1}}(e_{i_{\tau(d)}}^\ast  \cdots   e_{i_{\tau(1)}}^\ast)(e_{i_{\sigma^{-1}(d)}}  \cdots   e_{i_{\sigma^{-1}(1)}})\otimes \sigma\\[3pt]
&=\sum_{\sigma\in \mathfrak{S}_d}\varepsilon_{e_{i_d},\ldots,e_{i_1};\sigma^{-1}}\delta_{i_{\tau(d)}i_{\sigma^{-1}(d)}}\cdots \delta_{i_{\tau(1)}i_{\sigma^{-1}(1)}} \sigma\\[3pt]
&=\sum_{\sigma\in \mathfrak{S}_d}\varepsilon_{e_{i_d},\ldots,e_{i_1};\sigma^{-1}}\delta_{i_{\tau(d)}i_{(\tau\sigma)^{-1}\tau(d)}}\cdots \delta_{i_{\tau(1)}i_{(\tau\sigma)^{-1}\tau(1)}} \sigma\\[3pt]
&=\sum_{(\tau\sigma)^{-1}\in (\mathfrak{S}_d)_{I.\tau}}\varepsilon_{e_{i_d},\ldots,e_{i_1};\sigma^{-1}}\sigma,
\end{align*}where the fifth equality follows from Lemma \ref{Lemma 1}. Therefore \begin{align*}
\lefteqn{L(e_{i_{\tau(d)}})\cdots L(e_{i_{\tau(1)}})L(e_{i_{\tau(1)}}^\ast)\cdots L(e_{i_{\tau(d)}}^\ast)(e_{i_d}   \cdots  e_{i_1})}\\[3pt]
&=\sum_{(\tau\sigma)^{-1}\in (\mathfrak{S}_d)_{I.\tau}}\varepsilon_{e_{i_d},\ldots,e_{i_1};\sigma^{-1}}(e_{i_{\tau(d)}} \cdots  e_{i_{\tau(1)}}).\sigma\\[3pt]
&=\sum_{(\tau\sigma)^{-1}\in (\mathfrak{S}_d)_{I.\tau}}\varepsilon_{e_{i_d},\ldots,e_{i_1};\sigma^{-1}}(e_{i_{\sigma^{-1}(d)}} \cdots  e_{i_{\sigma^{-1}(1)}}).\sigma\\[3pt]
&=\sum_{(\tau\sigma)^{-1}\in (\mathfrak{S}_d)_{I.\tau}}((e_{i_{d}} \cdots  e_{i_{1}}).\sigma^{-1}).\sigma\\[3pt]
&=|(\mathfrak{S}_d)_{I.\tau}|e_{i_d}  \cdots  e_{i_1}\\[3pt]
&=|(\mathfrak{S}_d)_I|e_{i_d}  \cdots  e_{i_1}\\[3pt]
&=\mathbf{m}(I)!e_{i_d}  \cdots  e_{i_1},
\end{align*}
where the fifth equality follows by an isomorphism
$(\mathfrak{S}_d)_{I.\tau} \to (\mathfrak{S}_d)_I, \sigma \mapsto \tau\sigma\tau^{-1}$.
\end{proof}

\begin{cor}For any $I=(i_d,\ldots,i_1),J=(j_d,\ldots,j_1)\in [m+n]^d$, we have  \[L(e_{i_{d}})\cdots L(e_{i_{1}})L(e_{i_{1}}^\ast)\cdots L(e_{i_{d}}^\ast)(e_{j_d} \cdots e_{j_1})=\begin{cases}\mathbf{m}(I)!e_{i_d}  \cdots  e_{i_1}&\text{if }\mathbf{m}(I)= \mathbf{m}(J),\\[3pt]0&\text{if }\mathbf{m}(I)\neq \mathbf{m}(J).\end{cases}\]\end{cor}

We provide two more interesting formulas of such type which will be used in the sequel.

\begin{prop}For any $\varphi\in \overline{T}_d(V)$, we have \begin{equation}\sum_{i=1}^{m+n}L(e_{i})L(e_{i}^\ast)(\varphi)=d \varphi.\end{equation}\end{prop}
\begin{proof}For any $(j_d,\ldots,j_1)\in [m+n]^d$ and $\sigma\in \mathfrak{S}_\infty$, we have\begin{align*}
\sum_{i=1}^{m+n}L(e_{i})L(e_{i}^\ast)(e_{j_d}  \cdots  e_{j_1}\otimes \sigma)&=\sum_{i=1}^{m+n}\big(\rho_d(E_{ii})(e_{j_d}  \cdots  e_{j_1})\big)\otimes \sigma\\[3pt]
&=\rho_d\Bigg(\sum_{i=1}^{m+n}E_{ii}\Bigg)(e_{j_d}  \cdots  e_{j_1})\otimes \sigma\\[3pt]
&=d(e_{j_d}  \cdots  e_{j_1}\otimes \sigma).
\end{align*}Since each $\varphi\in \overline{T}_d(V)$ is a linear combination of $e_{j_d} \cdots  e_{j_1}\otimes \sigma$'s, we get the result.\end{proof}

We define the Euler operator
\begin{equation}
  A_d=\frac{1}{d!}\sum_{(i_d,\ldots,i_1)\in [m+n]^d}L(e_{i_{d}})\cdots L(e_{i_{1}})L(e_{i_{1}}^\ast)\cdots L(e_{i_{d}}^\ast).
\end{equation}

\begin{prop}\label{Euler operator}For any $\varphi\in \overline{T}_d(V)$, we have \begin{equation}A_d(\varphi)=\varphi.\end{equation}\end{prop}
\begin{proof}We have \begin{align*}
A_d(\varphi)&=\frac{1}{d!}\sum_{(i_d,\ldots,i_1)\in [m+n]^d}L(e_{i_{d}})\cdots L(e_{i_{1}})L(e_{i_{1}}^\ast)\cdots L(e_{i_{d}}^\ast)(\varphi)\\[3pt]
&=\frac{1}{d!}\sum_{(i_d,\ldots,i_2)\in [m+n]^{d-1}}L(e_{i_{d}})\cdots L(e_{i_{2}})\Bigg(\sum_{i_1=1}^{m+n} L(e_{i_{1}})L(e_{i_{1}}^\ast)\Bigg)\big(L(e_{i_{2}}^\ast)\cdots L(e_{i_{d}}^\ast)(\varphi)\big)\\[3pt]
&=\frac{1}{d!}\sum_{(i_d,\ldots,i_2)\in [m+n]^{d-1}}L(e_{i_{d}})\cdots L(e_{i_{2}})\big(L(e_{i_{2}}^\ast)\cdots L(e_{i_{d}}^\ast)(\varphi)\big)\\[3pt]
&=\frac{1}{d!}\sum_{(i_d,\ldots,i_3)\in [m+n]^{d-2}}L(e_{i_{d}})\cdots L(e_{i_{3}})\Bigg(\sum_{i_2=1}^{m+n} L(e_{i_{2}})L(e_{i_{2}}^\ast)\Bigg)\big(L(e_{i_{3}}^\ast)\cdots L(e_{i_{d}}^\ast)(\varphi)\big)\\[3pt]
&=\frac{1}{d!}\sum_{(i_d,\ldots,i_3)\in [m+n]^{d-2}}L(e_{i_{d}})\cdots L(e_{i_{3}})\big(2 L(e_{i_{3}}^\ast)\cdots L(e_{i_{d}}^\ast)(\varphi)\big)\\[3pt]
&=\cdots\\[3pt]
&=\frac{1}{d!}d!\varphi\\[3pt]
&=\varphi.
\end{align*}\end{proof}

\section{Generalizations of the Schur-Sergeev duality}\label{section:super-duality}
Let $W$ be a vector space and $A$ a subset of $\mathrm{End}_\mathbb{C}(W)$.
The commutant of $A$ is denoted by $\mathrm{End}_A(W)$.
That is $$\mathrm{End}_A(W)=\{f\in \mathrm{End}_\mathbb{C}(W)~|~fg=gf \text{ for all }g\in A\}.$$

For any $d\geq 1$ and $0\leq k\leq d$,
we consider $\mathrm{End}_\mathbb{C}(V^{\otimes d}\otimes _{\mathbb{C}\mathfrak{S}_d}\mathbb{C}\mathfrak{S}_{d+k})$.
Denote by $\mathcal{A}_d^k$ the subspace of $\mathrm{End}_\mathbb{C}(V^{\otimes d}\otimes _{\mathbb{C}\mathfrak{S}_d}\mathbb{C}\mathfrak{S}_{d+k})$
spanned by operators of the form $$L(w_d)\cdots L(w_1)L(\sigma)L(v_1^\ast)\cdots L(v_d^\ast),$$
with $w_1,\ldots,w_d\in V$, $\sigma\in \mathfrak{S}_{d+k}$, and $v_1^\ast,\ldots,v_k^\ast\in V^\ast$.
The representation $\pi_{d+k}$ of $\mathbb{C}\mathfrak{S}_{d+k}$ on $V^{\otimes d}\otimes _{\mathbb{C}\mathfrak{S}_d}\mathbb{C}\mathfrak{S}_{d+k}$ is given by right multiplication
and we identify $\pi_{d+k}(\mathbb{C}\mathfrak{S}_{d+k})$ with $\mathbb{C}\mathfrak{S}_{d+k}$.

\begin{thm}\label{super commutant theorem}
  We have
  \begin{equation*}
    \mathcal{A}_d^k =\mathrm{End}_{\mathbb{C}\mathfrak{S}_{d+k}}(V^{\otimes d}\otimes _{\mathbb{C}\mathfrak{S}_d}\mathbb{C}\mathfrak{S}_{d+k})\end{equation*} and 
\begin{equation*} \mathrm{End}_{\mathcal{A}_d^k}(V^{\otimes d}\otimes _{\mathbb{C}\mathfrak{S}_d}\mathbb{C}\mathfrak{S}_{d+k})=\mathbb{C}\mathfrak{S}_{d+k}.
  \end{equation*}
\end{thm}
\begin{proof}Apparently, by Proposition \ref{Commutant relation}, we have
  $\mathcal{A}_d^k\subset \mathrm{End}_{\mathbb{C}\mathfrak{S}_{d+k}}(V^{\otimes d}\otimes _{\mathbb{C}\mathfrak{S}_d}\mathbb{C}\mathfrak{S}_{d+k})$.
  Now for any $f\in \mathrm{End}_{\mathbb{C}\mathfrak{S}_{d+k}}(V^{\otimes d}\otimes _{\mathbb{C}\mathfrak{S}_d}\mathbb{C}\mathfrak{S}_{d+k})$ and $\varphi\in V^{\otimes d}\otimes _{\mathbb{C}\mathfrak{S}_d}\mathbb{C}\mathfrak{S}_{d+k}$, it follows from Proposition \ref{Euler operator} that \begin{align*}
f(\varphi)&=f(A_d(\varphi))\\
&=f\Bigg(\frac{1}{d!}\sum_{(i_d,\ldots,i_1)\in [m+n]^d}L(e_{i_{d}})\cdots L(e_{i_{1}})L(e_{i_{1}}^\ast)\cdots L(e_{i_{d}}^\ast)(\varphi)\Bigg)\\[3pt]
&=f\Bigg(\frac{1}{d!}\sum_{(i_d,\ldots,i_1)\in [m+n]^d}e_{i_{d}}  \cdots   e_{i_{1}}\otimes L(e_{i_{1}}^\ast)\cdots L(e_{i_{d}}^\ast)(\varphi)\Bigg)\\[3pt]
&=f\Bigg(\frac{1}{d!}\sum_{(i_d,\ldots,i_1)\in [m+n]^d}\pi_{d+k}\big(L(e_{i_{1}}^\ast)\cdots L(e_{i_{d}}^\ast)(\varphi)\big)(e_{i_{d}} \cdots   e_{i_{1}})\Bigg)\\[3pt]
&=\frac{1}{d!}\sum_{(i_d,\ldots,i_1)\in [m+n]^d}\pi_{d+k}\big(L(e_{i_{1}}^\ast)\cdots L(e_{i_{d}}^\ast)(\varphi)\big)f(e_{i_{d}} \cdots  e_{i_{1}})\\[3pt]
&=\frac{1}{d!}\sum_{(i_d,\ldots,i_1)\in [m+n]^d}f(e_{i_{d}}  \cdots   e_{i_{1}})L(e_{i_{1}}^\ast)\cdots L(e_{i_{d}}^\ast)(\varphi).
\end{align*}Write $f(e_{i_d} \cdots  e_{i_1})=\sum_{J=(j_d,\ldots,j_1)\in [m+n]^d}e_{j_d} \cdots   e_{j_1}\otimes t_J$ with $t_J\in \mathbb{C}\mathfrak{S}_{d+k}$. Then \begin{align*}
\lefteqn{f(e_{i_{d}}  \cdots   e_{i_{1}})L(e_{i_{1}}^\ast)\cdots L(e_{i_{d}}^\ast)(\varphi)}\\[3pt]
&=\sum_{J=(j_d,\ldots,j_1)\in [m+n]^d}L(e_{j_d})\cdots L( e_{j_1})L( t_J)L(e_{i_1}^\ast)\cdots L(e_{i_d}^\ast)(\varphi).
\end{align*}Hence $f\in \mathcal{A}_d^k$. As a consequence, $\mathrm{End}_{\mathbb{C}\mathfrak{S}_{d+k}}(V^{\otimes d}\otimes _{\mathbb{C}\mathfrak{S}_d}\mathbb{C}\mathfrak{S}_{d+k})\subset\mathcal{A}_d^k$.

Since $\mathbb{C}\mathfrak{S}_{d+k}$ is a semisimple algebra, it follows from the double commutant theorem that $\mathrm{End}_{\mathcal{A}_d^k}(V^{\otimes d}\otimes _{\mathbb{C}\mathfrak{S}_d}\mathbb{C}\mathfrak{S}_{d+k})=\mathbb{C}\mathfrak{S}_{d+k}$.\end{proof}

As an application of the above theorem, we provide a new proof of the Schur-Sergeev duality. Firstly, we need the following technical lemma.

\begin{lem}For any $i,j,k\in [m+n]$, we have \begin{equation*}L(e_i)L(e_j)L(e_k^\ast)=(-1)^{|e_i|(|e_j|+|e_k^\ast|)}\big(L(e_j)L(e_k^\ast)L(e_i)-\delta_{ik}L(e_j)\big),\end{equation*}on $\overline{T}_d(V)$.\end{lem}
\begin{proof}For any $(l_d,\ldots,l_1)\in [m+n]^d$ and $\sigma\in \mathfrak{S}_\infty$, we have  \begin{align*}
\lefteqn{(-1)^{|e_i|(|e_j|+|e_k^\ast|)}\big(L(e_j)L(e_k^\ast)L(e_i)-\delta_{ik}L(e_j)\big)(e_{l_d} \cdots  e_{l_1}\otimes \sigma)}\\[3pt]
&=(-1)^{|e_i|(|e_j|+|e_k^\ast|)}\big(\rho_d(E_{jk})(e_i  e_{l_d} \cdots  e_{l_1})\otimes \sigma-\delta_{ik}e_j  e_{l_d}  \cdots  e_{l_1}\otimes \sigma\big)\\[3pt]
&=(-1)^{|e_i|(|e_j|+|e_k^\ast|)}\big((E_{jk}e_i)e_{l_d} \cdots e_{l_1}\otimes \sigma-\delta_{ik}e_j  e_{l_d} \cdots  e_{l_1}\otimes \sigma\big)\\[3pt]
&\ \ \ \ +(-1)^{|e_i|(|e_j|+|e_k^\ast|)}\sum_{p=1}^d(-1)^{|E_{jk}|(|e_i|+|e_{l_d}|+\cdots+|e_{l_{p+1}}|)}e_i  e_{l_d} \cdots  (E_{jk}e_{l_p})  \cdots  e_{l_1}\otimes \sigma\\[3pt]
&=\sum_{p=1}^d(-1)^{|E_{jk}|(|e_{l_d}|+\cdots+|e_{l_{p+1}}|)}e_i  e_{l_d} \cdots  (E_{jk}e_{l_p})  \cdots  e_{l_1}\otimes \sigma\\[3pt]
&=L(e_i)\Bigg(\sum_{p=1}^d(-1)^{|E_{jk}|(|e_{l_d}|+\cdots+|e_{l_{p+1}}|)}e_{l_d} \cdots (E_{jk}e_{l_p})  \cdots  e_{l_1}\otimes \sigma\Bigg)\\[3pt]
&=L(e_i)(\rho_d(E_{jk})(e_{l_d} \cdots  e_{l_1})\otimes \sigma)
\\[3pt]
&=L(e_i)L(e_j)L(e_k^\ast)(e_{l_d} \cdots   e_{l_1}\otimes \sigma),
\end{align*}as desired.\end{proof}

\begin{cor}[Schur-Sergeev duality]\label{super duality theorem}
  For each $d\geq 1$, we have
  \begin{equation*}
    \rho_d(U(\mathfrak{gl}(m|n))) =\mathrm{End}_{\mathbb{C}\mathfrak{S}_{d}}(V^{\otimes d})\end{equation*}and\begin{equation*}
    \mathrm{End}_{\rho_d(U(\mathfrak{gl}(m|n)))}(V^{\otimes d}) =\mathbb{C}\mathfrak{S}_{d},
  \end{equation*}
  where $\rho_d$ denotes the representation of $U(\mathfrak{gl}(m|n))$ on $V^{\otimes d} \cong V^{\otimes d} \otimes_{\mathbb{C}\mathfrak{S}_{d}} \mathbb{C}\mathfrak{S}_{d}$.
\end{cor}
\begin{proof}We are going to show that $\mathcal{A}_d^0=\rho_d(U(\mathfrak{gl}(m|n)))$. We first notice that every element in $\mathcal{A}_d^0$ is a linear combination of operators of the form $$L(w_d)\cdots L(w_1)L(\sigma)L(v_1^\ast)\cdots L(v_d^\ast),$$with $w_1,\ldots,w_d\in V$, $\sigma\in \mathfrak{S}_{d}$, and $v_1^\ast,\ldots,v_k^\ast\in V^\ast$. By Proposition \ref{Relation 2}, the above operators can be rewritten in the form $$\eta L(w'_d)\cdots L(w'_1)L(v_1^\ast)\cdots L(v_d^\ast),$$ for some number $\eta$ and $w'_1,\ldots,w'_d\in V$. We identify $\varphi\in V^{\otimes d}\otimes _{\mathbb{C}\mathfrak{S}_d}\mathbb{C}\mathfrak{S}_{d}$ with $V^{\otimes d}$ in a natural way.

For any $\varphi\in V^{\otimes d}\otimes _{\mathbb{C}\mathfrak{S}_d}\mathbb{C}\mathfrak{S}_{d}$, by (\ref{Eij}), we have \begin{align*}
\rho_d(E_{ij})(\varphi)&=L(e_i)L(e_j^\ast)(\varphi)\\[3pt]
&=L(e_i)A_{d-1}L(e_j^\ast)(\varphi)\\[3pt]
&=\frac{1}{(d-1)!}\sum_{k_1,\ldots,k_{d-1}=1}^{m+n}L(e_i)L(e_{k_{d-1}})\cdots L(e_{k_1})L(e_{k_1}^\ast)\cdots L(e_{k_{d-1}}^\ast)L(e_j^\ast)(\varphi).
\end{align*}Therefore $\rho_d(E_{ij})\in \mathcal{A}_d^0$, and hence $\rho_d(U(\mathfrak{gl}(m|n)))\subset \mathcal{A}_d^0$.

On the other hand, for any $(i_p,\ldots,i_1), (j_p,\ldots,j_1)\in [m+n]^p$, we use induction on $p$ to show that $L(e_{i_p})\cdots L(e_{i_1})L(e_{j_1}^\ast)\cdots L(e_{j_p}^\ast)\in \rho_d(U(\mathfrak{gl}(m|n)))$. The case $p=1$ follows from (\ref{Eij}). Assume that the result holds for $p\geq 1$. By the above lemma, it is easy to see that\begin{align*}
\lefteqn{L(e_{i_{p+1}})\cdots L(e_{i_1})L(e_{j_1}^\ast)\cdots L(e_{j_{p+1}}^\ast)}\\[3pt]
&=\kappa L(e_{i_1})L(e_{j_1}^\ast) L(e_{i_{p+1}})\cdots L(e_{i_2})L(e_{j_2}^\ast)\cdots L(e_{j_{p+1}}^\ast)+\text{terms in }\mathcal{A}_{d-2}^0\\[3pt]
&=\kappa \rho_d(E_{i_1 j_1}) L(e_{i_{p+1}})\cdots L(e_{i_2})L(e_{j_2}^\ast)\cdots L(e_{j_{p+1}}^\ast)+\text{terms in }\mathcal{A}_{d-2}^0,
\end{align*}for some number $\kappa$. By the inductive hypothesis, we get the result.
\end{proof}

\section{Extensions through Hecke algebras}

\subsection{Hecke algebras}
Let $q$ be an indeterminate.
The Hecke algebra $\mathscr{H}_d$ of type $A_{d-1}$ is a $\mathbb{C}(q)$-algebra generated by $T_1, T_2,\ldots, T_{d-1}$ subject to
\begin{align*}
  &(T_i-q)(T_i+q^{-1})=0, \\
  &T_iT_{i+1}T_i=T_{i+1}T_iT_{i+1},\\
  &T_iT_j=T_jT_i,\qquad \mbox{for $|i-j|>1$}.
\end{align*}
For any reduced expression $s_{i_1}s_{i_2}\cdots s_{i_k}$ of $w \in\mathfrak{S}_d$,
we denote $T_w=T_{i_1}T_{i_2}\cdots T_{i_k}$. It is well-known that $T_w$ is independent of the choice of reduced expressions of $w$.

Let $\mathscr{H}_\infty$ be the inductive limit of the natural inclusions $\mathscr{H}_0\subset\mathscr{H}_1\subset\cdots$.
Given $k\geq 1$, we define an algebra homomorphism $ ^{\uparrow k}:\mathscr{H}_\infty\rightarrow \mathscr{H}_\infty$ by
$T_i^{\uparrow k}=T_{i+k}$.

Define a map $\gamma: [m+n] \times [m+n] \to \{-1, 1\}$ by
\[\gamma(i, j) =\left\{\begin{array}{ll}
    1 & \mbox{if $i > j$},\\
    -1 & \mbox{if $i \le j$}.
  \end{array}\right.\]

The following lemma will be used in the next section.
\begin{lem}\label{3t}
    For any $1 \le i, j, k \le m+n$, we have
$$ T_1^{\gamma(i, j)}T_2^{\gamma(i, k)}T_1^{\gamma(j, k)}
        = T_2^{\gamma(j, k)} T_1^{\gamma(i, k)}T_2^{\gamma(i, j)}.$$
\end{lem}
\begin{proof}
   It can be verified directly by the braid relation $T_1T_{2}T_1=T_{2}T_1T_{2}$ case by case.
\end{proof}

\subsection{$q$-tensor superspaces}
We introduce a parity on $[m+n]$ by
\[\widehat{i}=\left\{\begin{array}{ll}
  0 & \mbox{if $1\leq i\leq m$},\\
  1 & \mbox{if $m< i\leq m+n$}
\end{array}\right.\]
and let $q_{i} = q^{(-1)^{\widehat{i}}}$.

Let $\V=\V_{\overline{0}}\oplus\V_{\overline{1}}$ be a vector space over $\mathbb{C}(q)$ with a specified basis $\{e_1,e_2,\ldots,e_{m+n}\}$, where $\V_{\overline{0}}$ (resp. $\V_{\overline{1}}$) is spanned by $\{e_i~|~1\leq i\leq m\}$ (resp. $\{e_i~|~m+1\leq i\leq m+n\}$).

For $I=(i_d,\ldots,i_2,i_1)\in[m+n]^d$, we denote
\[e_{I}=e_{i_d}  \cdots  e_{i_2}  e_{i_1}.\]
Then $\{e_{I}~|~I\in[m+n]^d\}$ forms a basis of $\V^{\otimes d}$.

As before, $\mathfrak{S}_d$ acts on $[m+n]^d$ from the right hand side. The right $\mathscr{H}_d$-action on $\V^{\otimes d}$ is defined on the basis element $e_{I}$ as follows: for $1\leq k<d$,
\begin{equation}\label{action:Tk}
  e_{I}. T_k=
  \left\{
  \begin{array}{ll}
    (-1)^{\widehat{i_k}\widehat{i_{k+1}}}e_{I.s_k} & \mbox{if $i_k>i_{k+1}$},\\
    (-1)^{\widehat{i_k}} q_{i_k}e_{I} &\mbox{if $ i_k=i_{k+1}$},\\
    (-1)^{\widehat{i_k}\widehat{i_{k+1}}}e_{I.s_k}+(q-q^{-1})e_{I}& \mbox{if $i_k<i_{k+1}$}.
  \end{array}
  \right.
\end{equation}

We can rewrite \eqref{action:Tk} as
\begin{equation}\label{actionTk}
e_I . T_k^{\gamma(i_k, i_{k+1})} = (-1)^{\widehat{i_k}\widehat{i_{k+1}} } q_{i_k}^{-\delta_{i_{k}i_{k+1}}}e_{Is_k},
\end{equation}for $1\leq k<d$.
Since $\gamma(j, i) = -2\delta_{i, j} - \gamma(i, j)$, we have
\[
  e_I . T_k^{\gamma(i_{k+1}, i_k)} = e_I .( T_k^{-2\delta_{i_k i_{k+1}}} T_k^{-\gamma(i_k, i_{k+1})}) = q_{i_k}^{-2\delta_{i_k i_{k+1}}}e_I. T_k^{-\gamma(i_k, i_{k+1})},
\]for $1\leq k<d$.
The above equation together with the fact $T_k^{-1}=T_k+(q^{-1}-q)$ implies
\[
  e_I . T_k^{\gamma(i_{k+1}, i_k)} = e_I.(T_k^{-\gamma(i_k, i_{k+1})} - q^{\delta_{i_k, i_{k+1}}} + q^{-\delta_{i_k, i_{k+1}}}),
\]for $1\leq k<d$.

\subsection{Super-extensions}
We set
\[
  \widetilde{T}_d(\V)=\mathrm{Ind}_{\mathscr{H}_d}^{\mathscr{H}_\infty}\V^{\otimes d} = \V^{\otimes d} \otimes_{\mathscr{H}_d} \mathscr{H}_\infty,
\] and \[\widetilde{T}(\V)=\bigoplus_{d\geq0}\widetilde{T}_d(\V).\]
Endow an algebra structure on $\widetilde{T}(\V)$ whose product $\cdot$ is defined by
\[
    (e_{j_k}\cdots e_{j_1}\otimes T_\tau)\cdot(e_{i_d}\cdots e_{i_1}\otimes T_\sigma)
    =e_{j_k}\cdots e_{j_1} e_{i_d}\cdots e_{i_1}\otimes T_\tau^{\uparrow d}T_\sigma.
\]
The product $\varphi \cdot T$, with $\varphi\in\widetilde{T}(\V)$ and $ T\in\mathscr{H}_\infty$, provides a natural right $\mathscr{H}_\infty$-action on $\widetilde{T}(\V)$.

\subsection{Multiplications and Derivations}
For any $\varphi\in\widetilde{T}(\V)$, define
\begin{alignat*}{2}
  L(\varphi):\widetilde{T}(\V) &\longrightarrow &\widetilde{T}(\V),\\
  \psi ~ &\longmapsto &\varphi\cdot\psi.
\end{alignat*}

For $j\in[m+n]$, define
$L(e_j^*):\widetilde{T}(\V)\rightarrow \widetilde{T}(\V)$ by
\begin{eqnarray}
&&L(e_j^*)(e_{i_d}\cdots e_{i_1} \otimes T_\sigma)
  \nonumber\\\label{Lej*}
  &=&\sum_{k=1}^d (-1)^{\widehat{j}(\widehat{i_d}+\cdots+\widehat{i_{k+1}})} \inner{e_j^*, e_{i_k}}
  g_j(e_{i_d})\cdots g_j(e_{i_{k+1}})f_j(e_{i_{k-1}})\cdots f_j(e_{i_1})\otimes T_\sigma,
\end{eqnarray}
where $f_j:\V\rightarrow\V$ is the linear endomorphism defined by
\[f_j(e_r) = {q_j}^{-\delta_{j r}} e_r\]
and $g_j:\V \rightarrow \V \otimes \mathscr{H}_2$ is the linear map defined by
\[g_j(e_r)= e_r T_1^{-\gamma(j, r)}.\]
It is easy to see that $L(\varphi)$ is well-defined.
For $L(e_j^*)$, we only need to show that it commutes with the action of $\mathscr{H}_{\infty}$.
We introduce a linear map $h_k^{(j)}:\widetilde{T}_d(\V) \rightarrow \widetilde{T}_d(\V)$ for $1\leq k\leq d$ defined by
\[
    h_k^{(j)}(e_{i_d}\cdots e_{i_1}) = (-1)^{\widehat{j}(\widehat{i_d}+\cdots+\widehat{i_{k+1}})} \inner{e_j^*, e_{i_k}}
    g_j(e_{i_d})\cdots g_j(e_{i_{k+1}})f_j(e_{i_{k-1}})\cdots f_j(e_{i_1}).
\]
It is just the $k$-th term of $L(e_j^*)$ on $\widetilde{T}_d(\V)$ in \eqref{Lej*}.
For any $T \in \mathscr{H}_{\infty}$, we define $\pi(T) \in \mathrm{End}_{\mathbb{C}}(\widetilde{T}(\V))$ by $\pi (T)(\varphi)=\varphi \cdot T$ for $\varphi\in \widetilde{T}(\V)$.
To prove the well-definedness of $L(e_j^*)$, it suffices to show that $\pi(T_1), ..., \pi(T_{d-1})$ commutes with $\sum_{k=1}^d h_k^{(j)}$.
\begin{lem}
 Write $h_k$ for $h_k^{(j)}$. Then for $r = 1, ..., d-1$, we have
\begin{enumerate}
    \item[(i)] $h_k$ commutes with $\pi(T_r)$ unless $k = r, r+1$,
    \item[(ii)] $h_{r+1} + h_r$ commutes with $\pi(T_r)$.
\end{enumerate}
\end{lem}
\begin{proof}Fix $i_1, ..., i_d$.
  By the definition of $h_k$, we have the following relation between $h_k$ and $h_{k+1}$:
  \begin{equation}
    h_k = q_i^{2\delta_{j, i_{k+1}}} \pi(T_k^{-\gamma(i, i_{k+1})}) h_{k+1}  \pi(T_k^{\gamma(i_k, i_{k+1})}). \label{h-relation}
  \end{equation}
  Then when $r = k+1$ and $d \ge 3$, we have
  \begin{align*}
    &h_k \pi(T_{k+1}^{\gamma(i_{k+1}, i_{k+2})})\\
    &= q_i^{2(\delta_{j, i_{k+1}} + \delta_{j, i_{k+2}})}
    T_k^{-\gamma(i_k, i_{k+2})}
    \pi(T_{k+1}^{-\gamma(i_k, i_{k+1})}) h_{k+2}
    \pi(T_{k+1}^{\gamma(i_k, i_{k+1})} T_k^{\gamma(i_k, i_{k+2})} T_{k+1}^{\gamma(i_{k+1}, i_{k+2})})  \delta_{j, i_k}\\
    &= q_i^{2(\delta_{j, i_{k+1}} + \delta_{j, i_{k+2}})}
    \pi(T_{k+1}^{\gamma(i_{k+1}, i_{k+2})} T_k^{-\gamma(i_k, i_{k+1})} T_{k+1}^{-\gamma(i_k, i_{k+2})})
    h_{k+2}
    \pi(T_{k+1}^{\gamma(i_k, i_{k+2})} T_k^{\gamma(i_k, i_{k+1})})  \delta_{j, i_k} \\
    &= \pi(T_{k+1}^{\gamma(i_{k+1}, i_{k+2})}) h_k,
  \end{align*}
  where the first equality follows from \eqref{h-relation} and the second one from Lemma~\ref{3t}.

  When $k+2 \le r \le d-1$, we show that $\pi(T_r)$ commutes with $h_k$ by induction on $k$.
  We start the induction by $\pi(T_r) h_{r-1} = h_{r-1} \pi(T_r)$ which is the result of the last paragraph.
  For $k < r-1$, we have
  \begin{align*}
    \pi(T_r) h_k &= \pi(T_r) q_j^{2\gamma(j, i_{k+1})} \pi(T_k^{-\gamma(j, i_{k+1})}) h_{k+1}  \pi(T_k^{\gamma(i_k, i_{k+1})}) \delta_{j, i_k}\\ 
    &=  q_j^{2\delta_{j, i_{k+1}}} \pi(T_k^{-\gamma(j, i_{k+1})} T_r) h_{k+1}  \pi(T_k^{\gamma(i_k, i_{k+1})}) \delta_{j, i_k}\\
    &=  q_j^{2\delta_{j, i_{k+1}}} \pi(T_k^{-\gamma(j, i_{k+1})})  h_{k+1} \pi(T_k^{\gamma(i_k, i_{k+1})}) \pi(T_r) \delta_{j, i_k}\\
    &= h_k\pi(T_r),
  \end{align*}which proves (i).

  For (ii), notice that $$  h_{k+1}  \pi(T_k^{\gamma(i_k, i_{k+1})}) = q_j^{-2\delta_{j, i_{k+1}}} \pi(T_k^{\gamma(j, i_{k+1})}) h_k,$$ and
  \begin{align*}
    h_k \pi(T_k^{\gamma(i_k, i_{k+1})})
    &= q_j^{2\delta_{j, i_k}} \pi(T_k^{-\gamma(i_{k+1}, i_k)}) h_{k+1}  \pi(T_k^{\gamma(i_{k+1}, i_k)} T_k^{\gamma(i_k, i_{k+1})})\\
    &= q_j^{2\delta_{j, i_k}} \pi(T_k^{-\gamma(i_{k+1}, i_k)}) h_{k+1}  \pi(T_k^{-2\delta_{i_k, i_{k+1}}}) \\
    &= q_j^{2\delta_{j, i_k}} \pi(T_k^{-\gamma(i_{k+1}, i_k)}) h_{k+1}  {q_j}^{-2\delta_{i_k, i_{k+1}}}\\
    &= \pi(T_k^{-\gamma(i_{k+1}, i_k)}) h_{k+1} \\
    &=\pi(T_k^{\gamma(i_k, i_{k+1})} + q^{\delta_{i_k, i_{k+1}}} - q^{-\delta_{i_k, i_{k+1}}}) h_{k+1} \\
    &= \pi(T_k^{\gamma(i_k, i_{k+1})}) h_{k+1} + (q^{\delta_{i_k, i_{k+1}}} - q^{-\delta_{i_k, i_{k+1}}}) h_{k+1} \\
    &= \pi(T_k^{\gamma(i_k, i_{k+1})}) h_{k+1} + (-1)^{\widehat{i}} (q^{\delta_{i_k, i_{k+1}}} - q^{-\delta_{i_k, i_{k+1}}}) \pi(T_k^{\gamma(j, i_{k+1})}) q_j^{-\delta_{j, i_{k+1}}} h_k \\
    &= \pi(T_k^{\gamma(i_k, i_{k+1})}) h_{k+1} + (1 - q_j^{-2\delta_{j, i_{k+1}}}) \pi(T_k^{\gamma(j, i_{k+1})}) h_k \\
    &= \pi(T_k^{\gamma(i_k, i_{k+1})}) h_{k+1} + \pi(T_k^{\gamma(j, i_{k+1})}) h_k - q_j^{-2\delta_{j, i_{k+1}}} \pi(T_k^{\gamma(j, i_{k+1})}) h_k.
  \end{align*}
  Adding up the above two equations, we get the result.
\end{proof}

\subsection{General linear quantum supergroups}
The general linear quantum supergroup $U_q(\mathfrak{gl}(m|n))$ is a $\mathbb{C}(q)$-algebra generated by
\begin{equation*}
\left\{\begin{array}{lll}
\mbox{even generators:} & E_i,F_i,K_j^{\pm1}, & 1\leq i,j \leq m+n, i\neq m,m+n;\\
\mbox{odd generators:} & E_m,F_m &
\end{array}\right.
\end{equation*}
subject to the following relations:
 \begin{align*}
   & K_iK_j=K_jK_i, \quad K_iK_i^{-1}=K_i^{-1}K_i=1;\\
   & K_iE_j=q_i^{\delta_{ij}-\delta_{i,j+1}}E_jK_i,\quad K_iF_j=q_i^{\delta_{i,j+1}-\delta_{ij}}F_jK_i;\\
   &E_iF_j-(-1)^{\delta_{im}\delta_{jm}}F_jE_i=
   \delta_{ij}\frac{K_iK_{i+1}^{-1}-K_i^{-1}K_{i+1}}{q_i-q_i^{-1}};\\
   &E_iE_j=E_jE_i,\quad F_iF_j=F_jF_i, \quad (|i-j|>1);\\
   & E_i^2E_j-(q_i+q_i^{-1})E_iE_jE_i+E_jE_i^2=0,\quad (i\neq m, |i-j|=1);\\
    & F_i^2F_j-(q_i+q_i^{-1})F_iF_jF_i+F_jF_i^2=0,\quad (i\neq m, |i-j|=1);\\
    &E_m^2=F_m^2=E_mE_{m-1,m+2}+E_{m-1,m+2}E_m=F_mF_{m-1,m+2}+F_{m-1,m+2}F_m=0,
 \end{align*}
where
\begin{align*}
  E_{m-1,m+2}&=E_{m-1}E_mE_{m+1}-q^{-1}E_{m-1}E_{m+1}E_m-qE_mE_{m+1}E_{m-1}+E_{m+1}E_mE_{m-1},\\
  F_{m-1,m+2}&=F_{m-1}F_mF_{m+1}-q^{-1}F_{m-1}F_{m+1}F_m-qF_mF_{m+1}F_{m-1}+F_{m+1}F_mF_{m-1}.
\end{align*}

There is a Hopf superalgebra structure on $U_q(\mathfrak{gl}(m|n))$, whose comultiplication $\Delta$ is
\begin{align*}
  &\Delta(K_i)=K_i\otimes K_i,\\
  &\Delta(E_i)=1\otimes E_i+E_i\otimes K_iK_{i+1}^{-1},\\
  & \Delta(F_i)=F_i\otimes 1+K_{i+1}K_i^{-1}\otimes F_i.
\end{align*}Since $\Delta$ is coassociative, we can define
\[
  \Delta^{(d)} = (\Delta \otimes \mathrm{id}^{\otimes d-1}) \Delta^{(d-1)},
\]for $d\geq 2$. Here we set $\Delta^{(1)}=\Delta$ for convenience.

The following lemma can be derived by a straightforward computation.
\begin{lem}
  For $1 \le i \le m+n$, we have
  \begin{align*}
    \Delta^{(d)}(K_i)  &= K_i^{\otimes d+1},\\
    \Delta^{(d)}(E_i)  &= 1 \otimes \Delta^{(d-1)}(E_i) + E_i \otimes (K_iK_{i+1}^{-1})^{\otimes d},\\
    \Delta^{(d)}(F_i)  &= F_i \otimes 1^{\otimes d}  + K_{i+1}K_i^{-1}\otimes \Delta^{(d-1)}(F_i).
  \end{align*}
  Particularly, $|\Delta^{(d)}(E_i)| = |E_i|$ and $|\Delta^{(d)}(F_i)|=|F_i|$, where $|\cdot|$ means the parity of elements in $U_q(\mathfrak{gl}(m|n))$.
\end{lem}

The superspace $\V$ admits a natural representation of $U_q(\mathfrak{gl}(m|n))$ via
$$K_i.e_{j}=q_i^{\delta_{ij}}e_j,\quad E_i.e_j=\delta_{i+1,j}e_{i},\quad F_i.e_j=\delta_{i,j}e_{i+1}.$$
Then the $q$-tensor superspace $\V^{\otimes d}$ is also equipped with a $U_q(\mathfrak{gl}(m|n))$-module structure via $\Delta^{(d)}$.
We shall show that the $U_q(\mathfrak{gl}(m|n))$-module structure can be extended to $\widetilde{T}(\V)$ by $x.(e_{I}\otimes T_\sigma)=(x. e_{I})\otimes T_\sigma$ for any $x\in U_q(\mathfrak{gl}(m|n)), e_{I}\in \V^{\otimes d}, T_\sigma\in\mathscr{H}_\infty$.
It can be seen instantly for $K$, since
$$K_j.(e_{i_d} \cdots e_{i_1})=(q_{j})^{\delta_{j i_d} + \cdots + \delta_{j i_1}} e_{i_d}\cdots e_{i_1},$$
which commutes with $\mathscr{H}_{\infty}$.
The commutativity of $E$ and $F$ with $\mathscr{H}_{\infty}$ will be proved later (see \ref{uc}).

\section{Properties of operators in quantum version}

\subsection{Commutation relations}
  \begin{prop}\label{q-kc} For any $i, j \in[m+n]$ and $1\leq k<d$, as operators on $\widetilde{T}(\V)$, we have
\[
  K_j L(e_i) = q_j^{\delta_{ij}} L(e_i) K_j,\quad
  K_j L(e_i^*) = q_j^{-\delta_{ij}} L(e_i^*) K_j,\quad
  K_j L(T_k) = L(T_k) K_j.
\]
\end{prop}
\begin{proof}
  Check directly.
\end{proof}

\begin{lem}\label{q-cr-1}
  For $1 \le j \le m+n$,
  \[
    L(e_j) L(e_j^*) = \frac{K_j - K_j^{-1}}{q_j - q_j^{-1}}.
  \]
  \end{lem}
\begin{proof}
  For $I = (i_d,\ldots,i_2,i_1) \in [m+n]^d$, by a direct computation, we have
  \begin{align*}
    &\phantom{=} L(e_j) L(e_j^*) (e_{i_d}\cdots e_{i_1}\otimes T_\sigma ) \\
    &= \sum_{k=1}^{d}  (-1)^{\widehat{j}(\widehat{i_d}+\cdots+\widehat{i_{k+1}})} \inner{e_j^*, e_{i_k}}
    e_j g_j(e_{i_d})\cdots g_j(e_{i_{k+1}})f_j(e_{i_{k-1}})\cdots f_j(e_{i_1})\otimes T_\sigma \\
    &= \sum_{k=1}^{d}  \inner{e_j^*, e_{i_k}}
    e_j e_{i_d}(-1)^{\widehat{j} \widehat{i_d}} T_1^{-\gamma(j, i_d)}
    \cdots e_{i_{k+1}} (-1)^{\widehat{j} \widehat{i_{k+1}}} T_1^{-\gamma(j, i_{k+1})}
    f_j(e_{i_{k-1}})\cdots f_j(e_{i_1})\otimes T_\sigma \\
    &= \sum_{k=1}^{d}  q_{j}^{(\delta_{j i_d} + \cdots + \delta_{j i_{k+1}}) - (\delta_{j i_{k-1}}+ \cdots + \delta_{j i_1})}
    e_{i_d} \cdots e_{i_{k+1}} \inner{e_j^*, e_{i_k}} e_j e_{i_{k-1}} \cdots e_{i_1}\otimes T_\sigma \quad \mbox{(by \eqref{actionTk})}\\
    &= \sum_{k=1}^{d}  q_{j}^{(\delta_{j i_d} + \cdots + \delta_{j i_{k+1}}) - (\delta_{j i_{k-1}}+ \cdots + \delta_{j i_1})}
    \inner{e_j^*, e_{i_k}} e_{i_d} \cdots e_{i_{k+1}} e_{i_k} e_{i_{k-1}} \cdots e_{i_1}\otimes T_\sigma \\
    &= \Bigg(\sum_{k=1}^{d}  \delta_{j i_k} q_{j}^{(\delta_{j i_d} + \cdots + \delta_{j i_{k+1}}) - (\delta_{j i_{k-1}}+ \cdots + \delta_{j i_1})}\Bigg)
    e_{i_d} \cdots e_{i_{k+1}} e_{i_k} e_{i_{k-1}} \cdots e_{i_1}\otimes T_\sigma.
  \end{align*}
  Note that
  \begin{align*}
    &\sum_{k=1}^{d}  \delta_{j i_k} q_{j}^{(\delta_{j i_d} + \cdots + \delta_{j i_{k+1}}) - (\delta_{j i_{k-1}}+ \cdots + \delta_{j i_1})}\\
    &= \sum_{k=1}^{d}  \frac{q_{j}^{(\delta_{j i_d} + \cdots + \delta_{j i_{k+1}} + \delta_{j i_{k}}) - (\delta_{j i_{k-1}} + \cdots + \delta_{j i_1})} - q_{j}^{(\delta_{j i_d} + \cdots + \delta_{j i_{k+1}}) - (\delta_{j i_{k}} + \delta_{j i_{k-1}} + \cdots + \delta_{j i_1})}}{q_{j}-q_{j}^{-1}}\\
    &= \frac{q_{j}^{(\delta_{j i_d} + \cdots + \delta_{j i_1})} - q_{j}^{-(\delta_{j i_d} + \cdots + \delta_{j i_1})}}{q_{j}-q_{j}^{-1}}.
  \end{align*}
  We get the desired result.
\end{proof}

\begin{prop}\label{q-cr-2}
  For any $1 \le i,j \le m+n$, we have
  \begin{align}
    L(e_i) L(e_j) &= (-1)^{\widehat{i} \widehat{j}} q_i^{\delta_{ij}} L(e_j) L(e_i) L(T_1^{\gamma(i, j)}),\label{re1}\\
    L(e_i^*) L(e_j^*) &= (-1)^{\widehat{i} \widehat{j}} q_i^{\delta_{ij}} L(T_1^{\gamma(i, j)}) L(e_j^*) L(e_i^*),\label{re2}\\
    L(e_i^*) L(e_j) &= \delta_{ij} K_i^{-1} + (-1)^{\widehat{i} \widehat{j}} L(e_j) L(T_1^{-\gamma(i, j)}) L(e_i^*).\label{re3}
  \end{align}
\end{prop}
\begin{proof}
  The relation \eqref{re1} follows from
  \begin{align*}
    L(e_i) L(e_j) (e_{i_d}\cdots e_{i_1}) &= e_i e_j e_{i_d}\cdots e_{i_1}\\
    & = (-1)^{\widehat{i}\widehat{j} } q_{i}^{\delta_{ij}} e_j e_i T_1^{\gamma(i, j)} e_{i_d}\cdots e_{i_1}\\
    &= (-1)^{\widehat{i}\widehat{j} } q_{i}^{\delta_{ij}} L(e_j) L(e_i) L(T_1^{\gamma(i, j)}) e_{i_d}\cdots e_{i_1},
  \end{align*}
  while \eqref{re3} from
  \begin{align*}
    & L(e_i^*) L(e_j) (e_{i_d}\cdots e_{i_1})\\
    & = L(e_i^*) (e_j e_{i_d}\cdots e_{i_1})\\
    &=\langle e_i^*,e_j\rangle q_i^{-\delta_{ii_d}-\cdots-\delta_{ii_1}}e_{i_d}\cdots e_{i_1}\\
    &\qquad +\sum_{k=1}^d (-1)^{\widehat{i}(\widehat{j} + \widehat{i_d}+\cdots+\widehat{i_{k+1}})}
    \inner{e_j^*, e_{i_k}} g_i(e_{j}) g_i(e_{i_d})\cdots g_i(e_{i_{k+1}})f_i(e_{i_{k-1}})\cdots f_i(e_{i_1}) \\
    &= \inner{e_i^*, e_j} K_i^{-1}.e_{i_d}\cdots e_{i_1} \\
    &\qquad+ \sum_{k=1}^d (-1)^{\widehat{i}(\widehat{j} + \widehat{i_d}+\cdots+\widehat{i_{k+1}})}
    \inner{e_j^*, e_{i_k}} g_i(e_{j}) g_i(e_{i_d})\cdots g_i(e_{i_{k+1}})f_i(e_{i_{k-1}})\cdots f_i(e_{i_1}) \\
    &= \big(\delta_{ij} K_i^{-1} + (-1)^{\widehat{i} \widehat{j}} L(e_j) L(T_1^{-\gamma(i, j)}) L(e_i^*)\big) (e_{i_d}\cdots e_{i_1}).
  \end{align*}

Now let us prove \eqref{re2}, which is equivalent to \begin{equation}
  \label{re21}
  L(e_i^*) L(e_j^*)|_{\widetilde{T}_d(\V)}=(-1)^{\widehat{i} \widehat{j}} q_i^{\delta_{ij}} L(T_1^{\gamma(i, j)}) L(e_j^*) L(e_i^*)|_{\widetilde{T}_d(\V)},
\end{equation}for $d=0,1,2,\ldots$. We shall prove \eqref{re21} by induction on $d$. It is trivial for $d=0,1$. Take any $e_{i_d}\cdots e_{i_1}\otimes T_\sigma\in\widetilde{T}_d(\V)$. If $i_d=k\neq i,j$, then we have
\begin{align*}
&L(e_i^*) L(e_j^*)(e_{i_d}\cdots e_{i_1}\otimes T_\sigma)\\
&=L(e_i^*) L(e_j^*)L(e_k)(e_{i_{d-1}}\cdots e_{i_1}\otimes T_\sigma)\\
&=L(e_i^*)(-1)^{\widehat{j}\widehat{k}}L(e_k)L(T_1^{-\gamma(j,k)})L(e_j^*)(e_{i_{d-1}}\cdots e_{i_1}\otimes T_\sigma)\\
&=(-1)^{(\widehat{i}+\widehat{j})\widehat{k}}L(e_k)L(T_1^{-\gamma(i,k)})L(e_i^*)L(T_1^{-\gamma(j,k)})L(e_j^*)(e_{i_{d-1}}\cdots e_{i_1}\otimes T_\sigma)\\
&=(-1)^{(\widehat{i}+\widehat{j})\widehat{k}}L(e_k)L(T_1^{-\gamma(i,k)}T_2^{-\gamma(j,k)})L(e_i^*)L(e_j^*)(e_{i_{d-1}}\cdots e_{i_1}\otimes T_\sigma)\\
&=(-1)^{(\widehat{i}+\widehat{j})\widehat{k}+\widehat{i}\widehat{j}}q_i^{\delta_{ij}}L(e_k)L(T_1^{-\gamma(i,k)}T_2^{-\gamma(j,k)}T_1^{\gamma(i,j)})L(e_j^*)L(e_i^*)(e_{i_{d-1}}\cdots e_{i_1}\otimes T_\sigma)\\
&=(-1)^{(\widehat{i}+\widehat{j})\widehat{k}+\widehat{i}\widehat{j}}q_i^{\delta_{ij}}L(e_k)L(T_2^{\gamma(i,j)}T_1^{-\gamma(j,k)}T_2^{-\gamma(i,k)})L(e_j^*)L(e_i^*)(e_{i_{d-1}}\cdots e_{i_1}\otimes T_\sigma) \quad \mbox{(by \eqref{3t})}\\
&=(-1)^{(\widehat{i}+\widehat{j})\widehat{k}+\widehat{i}\widehat{j}}q_i^{\delta_{ij}}L(T_1^{\gamma(i,j)})L(e_k)L(T_1^{-\gamma(j,k)}T_2^{-\gamma(i,k)})L(e_j^*)L(e_i^*)(e_{i_{d-1}}\cdots e_{i_1}\otimes T_\sigma)\\
&=(-1)^{\widehat{i} \widehat{j}} q_i^{\delta_{ij}} L(T_1^{\gamma(i,j)}) L(e_j^*) L(e_i^*)L(e_k)(e_{i_{d-1}}\cdots e_{i_1}\otimes T_\sigma).
\end{align*}

If $i_d=i$, then
\begin{align*}
    &L(e_i^*) L(e_j^*) (e_{i_d} \cdots e_{i_1})\\
    &= L(e_i^*) L(e_j^*) L(e_i) (e_{i_{d-1}} \cdots e_{i_1})\\
    &= L(e_i^*) (\delta_{ij} K_j^{-1} + (-1)^{\widehat{i} \widehat{j}} L(e_i) L(T_1^{-\gamma(j, i)}) L(e_j^*))
    (e_{i_{d-1}} \cdots e_{i_1})\\
    &= (\delta_{ij} L(e_i^*) K_j^{-1} + (-1)^{\widehat{i} \widehat{j}} L(e_i^*) L(e_i) L(T_1^{-\gamma(j, i)}) L(e_j^*))
    (e_{i_{d-1}} \cdots e_{i_1})\\
    &= (\delta_{ij} L(e_i^*) K_j^{-1} + (-1)^{\widehat{i} \widehat{j}} K_i^{-1} L(T_1^{-\gamma(j, i)}) L(e_j^*)\\
    &\phantom{=}\ + (-1)^{\widehat{i} (\widehat{j} + \widehat{i})} L(e_i) L(T_1 T_2^{-\gamma(j, i)}) L(e_i^*)  L(e_j^*))
    (e_{i_{d-1}} \cdots e_{i_1})\\
    &= ((-1)^{\widehat{i} \widehat{j}} (L(T_1^{-\gamma(j, i)}) + (-1)^{\widehat{i} \widehat{j}}\delta_{ij} q_i^{-1}) K_i^{-1}  L(e_j^*)
     \\ &\phantom{=}\ + (-1)^{\widehat{i} (\widehat{j} + \widehat{i})} L(e_i) L(T_1 T_2^{-\gamma(j, i)}) L(e_i^*)  L(e_j^*))
    (e_{i_{d-1}} \cdots e_{i_1})\\
    &= ((-1)^{\widehat{i} \widehat{j}}
    (L(T_1^{\gamma(i, j)}) + \delta_{i j}(-1)^{\widehat{i}}(q_i - q_i^{-1}) + \delta_{ij}(-1)^{\widehat{i}} q_i^{-1}) K_i^{-1}  L(e_j^*)
    \\ &\phantom{=}\ + (-1)^{\widehat{i} (\widehat{j} + \widehat{i})} L(e_i) L(T_1 T_2^{-\gamma(j, i)}) L(e_i^*)  L(e_j^*))
    (e_{i_{d-1}} \cdots e_{i_1}) \\
    &= ((-1)^{\widehat{i} \widehat{j}} (L(T_1^{\gamma(i, j)}) + \delta_{i j}(-1)^{\widehat{i}}q_i) K_i^{-1}  L(e_j^*)
    \\ &\phantom{=}\  + (-1)^{\widehat{i} (\widehat{j} + \widehat{i})} L(e_i) L(T_1 T_2^{-\gamma(j, i)}) L(e_i^*)  L(e_j^*))
    (e_{i_{d-1}} \cdots e_{i_1})\\
    &= ((-1)^{\widehat{i} \widehat{j}} (L(T_1^{\gamma(i, j)})K_i^{-1}  L(e_j^*)) + \delta_{i j}q_i K_j^{-1} L(e_i^*)
    \\ &\phantom{=}\  + (-1)^{\widehat{i}} q_i^{\delta_{ij}} L(e_i) L(T_1 T_2^{-\gamma(j, i)} T_1^{\gamma(i, j)}) L(e_j^*) L(e_i^*))
    (e_{i_{d-1}} \cdots e_{i_1}),
\end{align*}
and
\begin{align*}
    & (-1)^{\widehat{i} \widehat{j}} q_i^{\delta_{ij}} L(T_1^{\gamma(i, j)}) L(e_j^*) L(e_i^*) L(e_i)\\
    &= (-1)^{\widehat{i} \widehat{j}} q_i^{\delta_{ij}} L(T_1^{\gamma(i, j)}) L(e_j^*) K_i^{-1}
    + (-1)^{\widehat{i} (\widehat{j} + \widehat{i})} q_i^{\delta_{ij}} L(T_1^{\gamma(i, j)}) L(e_j^*)  L(e_i) L(T_1) L(e_i^*)\\
    &= (-1)^{\widehat{i} \widehat{j}} L(T_1^{\gamma(i, j)}) K_i^{-1} L(e_j^*)
    + (-1)^{\widehat{i} (\widehat{j} + \widehat{i})} q_i^{\delta_{ij}} L(T_1^{\gamma(i, j)})
    (\delta_{ij} K_j^{-1}\\
     &\phantom{=}\ + (-1)^{\widehat{i} \widehat{j}} L(e_i) L(T_1^{-\gamma(j, i)}) L(e_j^*)) L(T_1) L(e_i^*)\\
    &= (-1)^{\widehat{i} \widehat{j}} L(T_1^{\gamma(i, j)}) K_i^{-1} L(e_j^*)
    + \delta_{ij} q_i^{\delta_{ij}} K_j^{-1} L(e_i^*)
    \\&\phantom{=}\ + (-1)^{\widehat{i}} q_i^{\delta_{ij}} L(e_i) L(T_2^{\gamma(i, j)} T_1^{-\gamma(j, i)} T_2) L(e_j^*) L(e_i^*) \\
    &= (-1)^{\widehat{i} \widehat{j}} L(T_1^{\gamma(i, j)}) K_i^{-1} L(e_j^*)
    + \delta_{ij} q_i K_j^{-1} L(e_i^*)
   \\&\phantom{=}\ + (-1)^{\widehat{i}} q_i^{\delta_{ij}} L(e_i) L(T_1 T_2^{-\gamma(j, i)} T_1^{\gamma(i, j)}) L(e_j^*) L(e_i^*) \quad \mbox{(by \ref{3t})}.
\end{align*}

If $i_d=j$, we can get the result similarly.
\end{proof}

\subsection{The representations of $U_q(\mathfrak{gl}(m|n))$ on $\widetilde{T}(\V)$}

For $1 \le i < m+n$ and $d\in\mathbb{N}$, we denote $$A_{i}^{(d)}=L(e_i)L(e_{i+1}^*)|_{\V^{\otimes d}},$$ and $$B_{i}^{(d)}=L(e_{i+1})L(e_i^*)|_{\V^{\otimes d}}.$$
Let $\sigma: \V \to \V$ be the linear endomorphism defined by $\sigma(e_k) = (-1)^{\widehat{k}}e_k$ for $1\leq k\leq m+n$.
\begin{lem}
We have
$$A_i^{(d)} = A_i^{(1)} \otimes K_{i+1}^{-1} + \sigma^{\delta_{im}} K_i^{-1} \otimes A_i^{(d-1)},$$ and
    $$B_i^{(d)} = B_i^{(1)} \otimes K_i + \sigma^{\delta_{im}} K_{i+1} \otimes B_i^{(d-1)}.$$
\end{lem}
\begin{proof}
  For $I = (i_{d-1},\ldots,i_2,i_1) \in [m+n]^{d-1}$ and $k \in [m+n]$, by direct calculation, we have
  \begin{align*}
      &A_i^{(d)}(e_k e_{i_{d-1}} \cdots e_{i_1})\\
      &=L(e_i)L(e_{i+1}^*)(e_k e_{i_{d-1}} \cdots e_{i_1})\\
      &= L(e_i)L(e_{i+1}^*)L(e_k) (e_{i_{d-1}} \cdots e_{i_1})\\
      &= L(e_i)(\delta_{i+1, k}K_{i+1}^{-1}
       + (-1)^{\widehat{i+1}\widehat{k}}L(e_k)L(T_1^{-\gamma(i+1, k)})L(e_{i+1}^*))(e_{i_{d-1}} \cdots e_{i_1})\\
      &= L(\delta_{i+1, k} e_i)\otimes K_{i+1}^{-1}(e_{i_{d-1}} \cdots e_{i_1})
      +(-1)^{\widehat{i+1}\widehat{k}}L(e_i)L(e_k)L(T_1^{-\gamma(i+1, k)})L(e_{i+1}^*)(e_{i_{d-1}} \cdots e_{i_1})\\
      &= A_i^{(1)}(e_k) \otimes  K_{i+1}^{-1} (e_{i_{d-1}} \cdots e_{i_1}) +
      (-1)^{\widehat{i+1}\widehat{k}}L(e_i)L(e_k)L(T_1^{-\gamma(i, k)-2\delta_{i, k}})L(e_{i+1}^*)(e_{i_{d-1}} \cdots e_{i_1})\\
      &= (A_i^{(1)} \otimes K_{i+1}^{-1}) (e_k e_{i_{d-1}} \cdots e_{i_1}) +
      (-1)^{\widehat{i+1}\widehat{k}}(-1)^{\widehat{i}\widehat{k}} q_{i}^{\delta_{i, k}}
      L(e_ke_i (T_1^{-2})^{\delta_{i, k}})L(e_{i+1}^*)(e_{i_{d-1}} \cdots e_{i_1})\\
      &= (A_i^{(1)} \otimes K_{i+1}^{-1}) (e_k e_{i_{d-1}} \cdots e_{i_1}) +
      (-1)^{\widehat{k}\delta_{i m}} q_{i}^{-\delta_{i, k}}
      L(e_k)L(e_i)L(e_{i+1}^*)(e_{i_{d-1}} \cdots e_{i_1})\\
      &= (A_i^{(1)} \otimes K_{i+1}^{-1}) (e_k e_{i_{d-1}} \cdots e_{i_1}) +
      (\sigma^{\delta_{i, m}} K_i^{-1})(e_k) \otimes A_i^{(d-1)} (e_{i_{d-1}} \cdots e_{i_1})\\
      &= (A_i^{(1)} \otimes K_{i+1}^{-1} + \sigma^{\delta_{i, m}} K_i^{-1} \otimes A_i^{(d-1)})(e_k e_{i_{d-1}} \cdots e_{i_1}).
  \end{align*}
The formula for $B_i^{(d)}$ can be proved in a similar way.
\end{proof}

\begin{lem}{\label{q-cr-3}}
  For $1 \le i < m+n$, we have
$$E_i K_i^{-1} = L(e_i)L(e_{i+1}^*)|_{\V^{\otimes d}},$$ and $$ K_i F_i = L(e_{i+1})L(e_i^*)|_{\V^{\otimes d}}.
$$
\end{lem}
\begin{proof}
 We only prove the first equation. The second one can be proved similarly. We use induction on $d$. It is clear for $d = 1$. Suppose the equation holds for $d=r-1$, then
  \begin{align*}
    &E_i K_i^{-1}(e_{i_r} e_{i_{r-1}} \cdots e_{i_1})\\
    &=(\Delta^{(r-1)}(E_i) \Delta^{(r-1)}(K_i^{-1}))(e_{i_r} e_{i_{r-1}} \cdots e_{i_1}) \\
    &=((1 \otimes \Delta^{(r-2)}(E_i) + E_i \otimes (K_iK_{i+1}^{-1})^{\otimes r})((K_i^{-1})^{\otimes r+1}))(e_{i_r} e_{i_{r-1}} \cdots e_{i_1})\\
    &=(K_i^{-1} \otimes \Delta^{(r-2)}(E_i K_i^{-1}) + E_i K_i^{-1} \otimes (K_{i+1}^{-1})^{\otimes r}) (e_{i_r} e_{i_{r-1}} \cdots e_{i_1})\\
    &= ((-1)^{|E_i|\widehat{i_r}} K_i^{-1} \otimes A_i^{(r-1)}+A_i^{(1)} \otimes K_{i+1}^{-1})(e_{i_r} e_{i_{r-1}} \cdots e_{i_1})\\
    &= (A_i^{(1)} \otimes K_{i+1}^{-1} + (-1)^{\delta_{im}\widehat{i_r}} K_i^{-1} \otimes A_i^{(r-1)})(e_{i_r} e_{i_{r-1}} \cdots e_{i_1})\\
    &= A_i^{(r)} (e_{i_r} e_{i_{r-1}} \cdots e_{i_1})\\
    & = L(e_i)L(e_{i+1}^*)(e_{i_r} e_{i_{r-1}} \cdots e_{i_1}),
  \end{align*}
  where the forth equality follows from the inductive hypothesis, and the sign comes from the super-action of $E_i K_i^{-1}$.\end{proof}

We further denote $E_{i,i+1}=E_i$, $E_{i+1,i}=F_i$, and
$$E_{ij}=E_{ik}E_{kj}-q_k E_{kj}E_{ik},\quad E_{ji}=E_{jk}E_{ki}-q_k^{-1} E_{ki}E_{jk},\quad (i<k<j).$$

\begin{thm}\label{q-cr-4}
  For any $1 \le i < j \le m+n$, we have
  \[E_{ij}K_i^{-1} = L(e_i)L(e_j^*)|_{\V^{\otimes d}},\] and \[K_j E_{ij} = L(e_j)L(e_i^*)|_{\V^{\otimes d}}.\]
\end{thm}
\begin{proof}As above, we only prove the first equation. We use induction on $j-i$. It has been showed in Lemma~\ref{q-cr-3} for the case $j-i = 1$.
  On $\V^{\otimes d}$, suppose $E_{ik}K_i^{-1} = L(e_i)L(e_k^*)$ for $ i < k < j$. Then
  \begin{align*}
    &E_{ij}K_i^{-1} = E_{ij}K_i^{-1} K_{j-1}^{-1} K_{j-1}\\
    &= (E_{i(j-1)}E_{(j-1)j} - q_{j-1} E_{(j-1)j} E_{i(j-1)})(K_i^{-1}K_{j-1}^{-1})K_{j-1}\\
    &= (E_{i(j-1)}K_i^{-1}E_{(j-1)j} K_{j-1}^{-1} - q_{j-1} E_{(j-1)j} E_{i(j-1)} K_i^{-1}K_{j-1}^{-1})K_{j-1}\\
    &= (L(e_i)L(e_{j-1}^*) L(e_{j-1})L(e_j^*)-
     q_{j-1}E_{(j-1)j} L(e_i)L(e_{j-1}^*) K_{j-1}^{-1})K_{j-1}\\
    &= (L(e_i)L(e_{j-1}^*) L(e_{j-1})L(e_j^*)-E_{(j-1)j}K_{j-1}^{-1}) L(e_i)L(e_{j-1}^*)K_{j-1}.
  \end{align*}
  Moreover,
  \begin{align*}
    &(E_{(j-1)j} K_{j-1}^{-1})L(e_i)L(e_{j-1}^*)\\
    & = L(e_{j-1}) L(e_j^*) L(e_i) L(e_{j-1}^*)\\
    &= (-1)^{\widehat{i} \widehat{j}} L(e_{j-1}) L(e_i)L(T_1^{-1}) L(e_j^*) L(e_{j-1}^*)\\
    &= (-1)^{\widehat{i} \widehat{j} + \widehat{i} \widehat{j-1} + \widehat{j} \widehat{j-1}}
    L(e_i) L(e_{j-1}) L(T_1) L(e_{j-1}^*) L(e_j^*) \\
    &= (-1)^{\widehat{i} \widehat{j} + \widehat{i} \widehat{j-1} + \widehat{j} \widehat{j-1} + \widehat{j-1} \widehat{j-1}}
    L(e_i) (L(e_{j-1}^*) L(e_{j-1}) - K_{j-1}^{-1}) L(e_j^*) \\
    &= (-1)^{(\widehat{i} + \widehat{j-1})( \widehat{j-1} + \widehat{j})}
    (L(e_i) L(e_{j-1}^*) L(e_{j-1}) L(e_j^*) - L(e_i) L(e_j^*) K_{j-1}^{-1})\\
    &= L(e_i) L(e_{j-1}^*) L(e_{j-1}) L(e_j^*) - L(e_i) L(e_j^*) K_{j-1}^{-1}.
  \end{align*}
  So we get the result.
\end{proof}

\begin{rem}\label{uc}
  Based on these identities on $\V^{\otimes d}$ and the fact that $L(e), L(e^*), K, K^{-1}$ all commute with $\mathscr{H}_{\infty}$,
  we can actually have a well-defined $U_q(\mathfrak{gl}(m|n))$-module structure on $\widetilde{T}(\V)$ and get the same result.
\end{rem}

Furthermore, we can get the following proposition.
\begin{prop}\label{q-final}
  For any $v_1, ...,v_k \in \V$, $v_1^*, ...,v_k^* \in \V^*$, we have
  \[
    L(v_k) \cdots L(v_1) L(v_1^*) \cdots L(v_k^*) \in \rho(U_q(\mathfrak{gl}(m|n))),
  \]
  where $\rho$ denotes the representation of $U_q(\mathfrak{gl}(m|n))$ on $\widetilde{T}(\V)$.
\end{prop}

\begin{proof}
  We first claim that
  \[
    L(e_i) L(e_j) L(e_k^*) = a L(e_i) L(e_k^*) L(e_j) + b L(e_j) L(e_k^*) L(e_i) + c K_k^{-1} L(e_i) + d K_k^{-1} L(e_j),
  \]
  where $a, b, c, d \in \mathbb{C}(q)$.
  Recall from Propositions~\ref{q-kc} and \ref{q-cr-2} that
  \begin{align*}
   L(e_i)L(e_j)L(T^{-\gamma(k, j)})L(e_k^*)&= (-1)^{\widehat{j} \widehat{k}}(L(e_i) L(e_k^*) L(e_j) - \delta_{j k} L(e_i) K_j^{-1})\\
      &= (-1)^{\widehat{j} \widehat{k}}L(e_i) L(e_k^*) L(e_j) + \delta_{j k}  q_k^{-\delta_{i, k}} K_k^{-1}L(e_i).
  \end{align*}
  Then, if $i = j$, we have
  \begin{align*}
L(e_i)L(e_j)L(e_k^*)&= (-1)^{\widehat{i}} q_i^{\gamma(k, i)} L(e_i)L(e_i)(L(T^{-\gamma(k, i)}))L(e_k^*)\\
      &= q_i^{\gamma(k, i)} L(e_i) (L(e_k^*)L(e_i) - \delta_{i k} K_i^{-1})\\
      &= q_i^{\gamma(k, i)} L(e_i) L(e_k^*)L(e_i) - \delta_{i k} K_k^{-1} L(e_i).
  \end{align*}
  If $i \neq j$, we have
  \begin{align*}
      & L(e_i)L(e_j)L(e_k^*) \\
      &= (-1)^{\widehat{i}\widehat{j}}L(e_j)L(e_i)L(T^{\gamma(i, j)})L(e_k^*)\\
      &= (-1)^{\widehat{i}\widehat{j}}L(e_j)L(e_i)L(T^{-\gamma(k, i)} + c_1)L(e_k^*)\\
      &= (-1)^{\widehat{i}\widehat{j}}L(e_j)L(e_i)L(T^{-\gamma(k, i)})L(e_k^*) + c_1 L(e_i)L(e_j)L(T^{\gamma(j, i)})L(e_k^*)\\
      &= (-1)^{\widehat{i}\widehat{j}}L(e_j)L(e_i)L(T^{-\gamma(k, i)})L(e_k^*) + c_1 L(e_i)L(e_j)L(T^{-\gamma(k, j)} + c_2)L(e_k^*)\\
      &= (-1)^{\widehat{i}\widehat{j}}L(e_j)L(e_i)L(T^{-\gamma(k, i)})L(e_k^*) + c_1 L(e_i)L(e_j)L(T^{-\gamma(k, j)})L(e_k^*) + c_1c_2 L(e_i)L(e_j)L(e_k^*),
  \end{align*}
  where $c_1, c_2$ is $\pm(q-q^{-1})$ or $0$.
  By direct calculation, we have $c_1=0$ or $c_2=0$ (i.e., $c_1c_2=0$), since $\gamma(i, j) = -\gamma(k, i)$ or $\gamma(j, i) = -\gamma(k, j)$.
  Then we get the claim.

  From the claim, we can see that
  \begin{align*}
    & L(e_{i_k}) \cdots L(e_{i_1}) L(e_{i_1}^*) \cdots L(e_{i_k}^*)\\
    &= a L(e_{i_{k}}) L(e_{i_k}^*)  \cdots \  + b L(e_{i_{k-1}}) L(e_{i_k}^*)  \cdots \  + c K_k^{-1} \cdots
  \end{align*}
  where $\cdots$ are elements of the form $L(e_{i_{j-1}}) \cdots L(e_{j_1}) L(e_{j_1}^*) \cdots L(e_{j_{k-1}}^*)$ and  $a, b, c \in \mathbb{C}(q)$.
  Continuing this procedure, we get the desired result.
\end{proof}

\section{Generalized quantum Schur-Sergeev duality}

\subsection{Euler operators}
For any nonnegative integer $k$, we denote by
  \[
    [k] = \frac{q^k - q^{-k}}{q - q^{-1}}\quad \mbox{and} \quad [k]! = [1][2] \cdots [k]
  \]
the $q$-quantum integers and $q$-quantum factorials, respectively.

For $i \in [m+n], a \in \mathbb{Z}$, set
\[
  [K_i:a] = \frac{q_i^a K_i - q_i^{-a} K_i^{-1}}{q_i - q_i^{-1}},
\]
\[
  [K_i]^h_! = [K_i:0][K_i:-1] \cdots [K_i:1-h], \quad \mathrm{for} \ h \in \mathbb{N_+},
\]
and $[K_i]^0_! = 1$.
Then by Lemma~\ref{q-cr-1} we have
\begin{equation}
  \underbrace{L(e_i) \cdots L(e_i)}_{h} \underbrace{L(e_i^*) \cdots L(e_i^*)}_{h} = [K_i]^h_!. \label{K_1}
\end{equation}

Set $\mathcal{I} = \{(i_d, i_{d-1}, ..., i_1) \in [m+n]^d~|~ i_d \le i_{d-1} \le \cdots \le i_1\}$.
For any $I \in \mathcal{I}$, let $[\mathbf{m}(I)]! = [\mathbf{m}_1(I)]! [\mathbf{m}_2(I)]! \cdots [\mathbf{m}_{m+n}(I)]!$
where $\mathbf{m}_k(I)$ is the multiplicity of $k$ in $I$.
By \eqref{K_1} and Proposition~\ref{q-kc} we have
\begin{equation}
    L(e_{i_d}) \cdots L(e_{i_1}) L(e_{i_1}^*) \cdots L(e_{i_d}^*) = \prod_{j \in [m+n]} [K_j]^{\mathbf{m}_j(I)}_!
    ,\quad I = (i_d, i_{d-1}, ..., i_1) \in \mathcal{I}. \label{K_2}
\end{equation}

Considering the Euler operator
\[
  A_d = \sum_{I \in \mathcal{I}} \frac{1}{[\mathbf{m}(I)]!} L(e_{i_d}) \cdots L(e_{i_1}) L(e_{i_1}^*) \cdots L(e_{i_d}^*),
\]
we have
\begin{lem}
  For any $\phi \in \widetilde{T}(\V)$,
  \[
    A_d(\phi) = \phi.
  \]
\end{lem}
\begin{proof}
  For any $I, J \in \mathcal{I}$, if $\mathbf{m}_i(I) = \mathbf{m}_i(J)$, we have
  \[
      [K_i]^{\mathbf{m}_i(I)}_!.e_{J} = [\mathbf{m}_i(I)]!e_{J},
  \]
  and if $\mathbf{m}_i(I) > \mathbf{m}_i(J)$, we have
  \[
      [K_i]^{\mathbf{m}_i(I)}_!.e_{J} = 0.
  \]
  In the case where $\mathbf{m}_i(I) < \mathbf{m}_i(J)$ for some $i$,
  there must exist some index $j$ such that $\mathbf{m}_j(I) > \mathbf{m}_j(J)$. Thus by \eqref{K_2}, one has
  \[
    L(e_{i_d}) \cdots L(e_{i_1}) L(e_{i_1}^*) \cdots L(e_{i_d})(e_J) = \delta_{I, J} \Bigg(\prod_{i \in [m+n]} [\mathbf{m}_i(I)]!\Bigg) e_J = \delta_{I, J} [\mathbf{m}(I)]! e_J.
  \]

  Now for any $L \in [m+n]^d$,
  by \eqref{actionTk}, there exists $J \in \mathcal{I}$ and $T\in \mathscr{H}_d$ such that $e_L = e_J .T$.
  Since $L(e_i)$ and $L(e_i^*)$ commute with $T$, we have
  \[
      A_d(e_L) = A_d(e_{J}). T
      = \sum_{I \in \mathcal{I}} \frac{1}{[\mathbf{m}(I)]!} \delta_{I, J} [\mathbf{m}(I)]! e_{J}. T = e_J. T = e_{L}.
  \]\end{proof}

\subsection{Descriptions of commutants}

For any $d \ge 1$ and $0 \le k \le d$, we denote by $\mathcal{A}_d^k$ the subspace of
$\mathrm{End}_{\mathbb{C}(q)}(\V^{\otimes d} \otimes _{\mathscr{H}_d} \mathscr{H}_{d+k})$
spanned by operators of the form
\[
  L(w_d)\cdots L(w_1)L(\sigma)L(v_1^*)\cdots L(v_d^*),
\]
with $w_1,\ldots,w_d\in \V$, $\sigma\in \mathscr{H}_{d+k}$,  and $v_1^*,\ldots,v_k^* \in \V^*$.

The representation $\pi_{d+k}$ of $\mathscr{H}_{d+k}$ on $\V^{\otimes d} \otimes _{\mathscr{H}_d} \mathscr{H}_{d+k}$ is given by right multiplication
and we identify $\pi_{d+k}(\mathscr{H}_{d+k})$ with $\mathscr{H}_{d+k}$.


\begin{thm}\label{g-duality}
  Asumme that $[d]! \neq 0$. We have
  \begin{equation*}
    \mathcal{A}_d^k = \mathrm{End}_{\mathscr{H}_{d+k}}(\V^{\otimes d} \otimes _{\mathscr{H}_d} \mathscr{H}_{d+k}),\end{equation*}and
    \begin{equation*}\mathrm{End}_{\mathcal{A}_d^k}(\V^{\otimes d} \otimes _{\mathscr{H}_d} \mathscr{H}_{d+k}) =\mathscr{H}_{d+k}.
  \end{equation*}
\end{thm}
\begin{proof}
  It is easy to see that $\mathcal{A}_d^k \subset \mathrm{End}_{\mathscr{H}_{d+k}}(\V^{\otimes d} \otimes _{\mathscr{H}_d} \mathscr{H}_{d+k})$.
 Since $\mathcal{H}_d(q)$ is semisimple, by  the dobule commutant theorem, it suffices to show
  $\mathrm{End}_{\mathscr{H}_{d+k}}(\V^{\otimes d} \otimes _{\mathscr{H}_d} \mathscr{H}_{d+k}) \subset \mathcal{A}_d^k$.

  For $f \in \mathrm{End}_{\mathscr{H}_{d+k}}(\V^{\otimes d} \otimes _{\mathscr{H}_d} \mathscr{H}_{d+k}), \phi \in \V^{\otimes d} \otimes _{\mathscr{H}_d} \mathscr{H}_{d+k}$, we have
    \begin{align*}
    f(\phi)& = f(A_d(\phi))\\
    &= f\left(\sum_{I \in \mathcal{I}} \frac{1}{[\mathbf{m}(I)]!} L(e_{i_d}) \cdots L(e_{i_1}) L(e_{i_1}^*) \cdots L(e_{i_d}^*) \phi\right)\\
      &= f\left(\sum_{I \in \mathcal{I}} \frac{1}{[\mathbf{m}(I)]!} e_{i_d} \cdots e_{i_1} \tau_{I}\right)\\
      &= \sum_{I \in \mathcal{I}} \frac{1}{[\mathbf{m}(I)]!} f(e_{i_d} \cdots e_{i_1}) \tau_{I}\\
      &= \sum_{I \in \mathcal{I}} \frac{1}{[\mathbf{m}(I)]!} L(f(e_{i_d} \cdots e_{i_1})) L(e_{i_1}^*) \cdots L(e_{i_d}^*) \phi,
    \end{align*}
    where $\tau_{I} = L(e_{i_1}^*) \cdots L(e_{i_d}^*) \phi \in \mathscr{H}_{d+k}$.
    So\[f = \sum_{I \in \mathcal{I}} \frac{1}{[\mathbf{m}(I)]!} L(f(e_{i_d} \cdots e_{i_1})) L(e_{i_1}^*) \cdots L(e_{i_d}^*).\]
    Since $L(f(e_{i_d} \cdots e_{i_1})) \in \V^{\otimes d} \otimes _{\mathscr{H}_d} \mathscr{H}_{d+k}$,
    we have $f \in \mathcal{A}_d^k$.
\end{proof}

\begin{cor}[Mitsuhashi]\label{q-Schur duality}
  Asumme that $[d]! \neq 0$. For each $d \ge 1$, we have
  \begin{equation*}
    \rho_d(U_q(\mathfrak{gl}(m|n))) = \mathrm{End}_{\mathscr{H}_{d}}(\V^{\otimes d}),\end{equation*}and
    \begin{equation*}\mathrm{End}_{\rho_d(U_q(\mathfrak{gl}(m|n)))}(\V^{\otimes d}) =\mathscr{H}_{d},
  \end{equation*}
  where $\rho_d$ denotes the representation of $U_q(\mathfrak{gl}(m|n))$ on $\V^{\otimes d} \cong \V^{\otimes d} \otimes _{\mathscr{H}_d} \mathscr{H}_{d}$.
\end{cor}

\begin{proof}
  We just need to prove that $\rho_d(U_q(\mathfrak{gl}(m|n))) = \mathcal{A}_d^0$.

  Proposition~\ref{q-final} shows that $\mathcal{A}_d^0 \subset \rho_d(U_q(\mathfrak{gl}(m|n)))$. The other side is proved as follows.
  We first notice that $K_i \in \mathrm{End}_{\mathscr{H}_{d}}(\V^{\otimes d})$. In the proof of Theorem~\ref{g-duality}, we have seen that
  $K_i \in \mathcal{A}_d^0$. Actually, we have
  \[
    K_i = \sum_{I \in \mathcal{I}} \frac{1}{[\mathbf{m}(I)]!}
    q_{j}^{\delta_{i i_d} + \cdots + \delta_{i i_1}} L(e_{i_d}) \cdots L(e_{i_1}) L(e_{i_1}^*) \cdots L(e_{i_d}^*).
  \]
  It follows from Theorem~\ref{q-cr-4} that $E_{ij} = L(e_i)L(e_j^*) K_i$ and $ E_{ij} = K_j^{-1} L(e_j)L(e_i^*)$ for $i < j$ on ${\V^{\otimes d}}$. So these elements commute with $\mathcal{B}_{d}$, and hence $\rho_d(U_q(\mathfrak{gl}(m|n))) \subset \mathcal{A}_d^0$.
\end{proof}


\end{document}